\documentclass{amsart}

\input xy
\xyoption{all}

\usepackage{amssymb}

\usepackage[bookmarks=false]{hyperref}

\newtheorem{theorem}[equation]{Theorem}
\newtheorem{lemma}[equation]{Lemma}
\newtheorem{proposition}[equation]{Proposition}
\newtheorem{corollary}[equation]{Corollary}
\newtheorem{definition}[equation]{Definition}

\theoremstyle{definition}
\newtheorem{example}[equation]{Example}

\numberwithin{equation}{subsection}

\def\mc{\mathcal}
\def\mb{\mathbb}
\def\Spec{\textnormal{Spec }}
\def\Hom{\textnormal{Hom}}
\def\sheafHom{\mathcal{H}om}
\def\Ob{\textbf{Ob }}
\def\Aut{\textnormal{Aut}}

\def\mod{\textbf{Mod}}
\def\lmod{\textbf{LMod}}

\begin{document}

\title{Equicharacteristic \'etale cohomology in dimension
one}

\author{Carl A.~Miller}
\email{carlmi@umich.edu}
\address{Department of Mathematics \\ 530 Church St. \\ University of Michigan \\
Ann Arbor, MI 48109}
\urladdr{http://www.umich.edu/~carlmi/}


\thanks{This paper was written while the author
was supported by NSF grant 
DMS-0502170.}

\begin{abstract}
The Grothendieck-Ogg-Shafarevich formula expresses the Euler
characteristic of an \'etale sheaf on a curve in terms of
local data.  The purpose of this paper is to prove
a version of the  G-O-S formula which applies to
equicharacteristic sheaves (a bound, rather than an equality).
This follows
a proposal of R.~Pink.

The basis for the result is the characteristic-$p$
``Riemann-Hilbert'' correspondence, which relates
equicharacteristic \'etale sheaves to $\mc{O}_{F, X}$-modules.
In the paper we prove a version of 
this correspondence for curves,
considering both local and global settings.
In the process we define an invariant,
the ``minimal root index,'' which measures
the local complexity of an $\mc{O}_{F, X}$-module.
This invariant provides the local terms for the main result.
\end{abstract}

\maketitle

\section{Introduction}

\subsection{Overview of paper}

We are concerned with computing sizes of \'etale cohomology
groups in positive characteristic.  Fix a base field, $k$, which
is algebraically closed and has characteristic $p > 0$.

In general, the properties of \'etale cohomology groups over $k$
depend heavily on what coefficient ring one chooses
to use.  One can assume that the coefficient ring is
$p$-torsion (or $p^r$-torsion), or one can assume that the coefficient ring 
has a characteristic which is coprime to $p$.
It is typical to separate these two cases.  The first
case (the ``equicharacteristic'' case) seems to be less
tractable than the second.  But a number of motivations
for studying the second case do carry over to the first.
For example, it is known that zeta functions modulo $p$
can be computed from equicharacteristic \'etale cohomology
groups.  (See \cite{katzsga}.)

A good tool for computing sizes of \'etale cohomology groups
is the ``Grothen\-dieck-Ogg-Shafarevich formula.''  Let $Y$
be a smooth projective $k$-curve, and let $N$ be a constructible
\'etale sheaf of $\mathbb{F}_\ell$-modules on $Y$, where $\ell$
denotes a prime different from $p$.  Assume that the sheaf $N$
is locally constant on an open subset $X \subseteq Y$, and that
its stalks are zero at points outside of $X$.  Then,
\begin{eqnarray}
\chi ( Y , N ) = (2 - 2g) n - \sum_{y \in \left| Y \smallsetminus
X \right|} \left( n + Sw_y ( N ) \right).
\end{eqnarray}
In this formula, $\chi ( Y, N )$ denotes the Euler characteristic 
of $N$ (the alternating sum of the dimensions $h^i ( Y, N)$), $g$
denotes the genus of $Y$, and $n$ denotes the generic rank
of $N$.  The expression $Sw_y ( N )$ denotes the ``Swan conductor,''
an invariant which measures the local ramification of the sheaf $N$.
(See \cite{raynaud} for a discussion of this formula, including
an application to surfaces.)

It is natural to ask whether a similar formula could be
constructed for the equicharacteristic case.  This question
leads to some difficulties, however, which are tied up
with the unpredictable behavior of $H^1 ( Y, N )$.  It is possible 
to construct an example of two $\mathbb{F}_p$-sheaves $N$
and $N'$ on the same curve $Y$, both having the same rank
and local ramification, but which nonetheless have different
Euler characteristics.
(See Example~\ref{ellipticcurveexample}
in this paper.)
So clearly, an exact formula for $\chi ( Y, N)$ based on
local information about $N$ will not be possible.\footnote{An
exact formula may be possible if we allow global data for the
sheaf $N$.
The work of W.~A.~Hawkins in \cite{hawkins1}
and \cite{hawkins2} contains results in that direction.}

A good compromise is to construct a {\it lower bound} for
$\chi ( Y, N)$ in the equicharacteristic case.  This
idea was proposed by R.~Pink.  Pink himself
proved a lower bound for $\chi ( Y, N )$ which applies 
under some restrictions on the wild ramification of $N$
(Theorem~0.2 in \cite{eulerpink}).
The purpose of this paper is to prove the following
general extension of Pink's theorem.
\begin{theorem}
\label{mainthmquoted}
Let $Y \to \Spec k$ be a smooth projective $k$-curve
of genus $g$.
Let $M$ be a rank-$n$ constructible \'etale sheaf of $\mathbb{F}_p$-modules
on $Y$ in which all sections have open support.  Then,
\begin{eqnarray}
\label{mainthmquotedformula}
\chi ( Y, M ) \geq (1 - g)n - \sum_{y \in Y(k) }
\mathfrak{C} \left( M_{(y)} \right).
\end{eqnarray}
\end{theorem}
In the expression above, $\mathfrak{C} \left( M_{(y)} \right)$
denotes a local invariant which we call the ``minimal root index.''

The proof of Theorem~\ref{mainthmquoted} is based on a study
of the relationship between \'etale $\mathbb{F}_p$-sheaves and
quasi-coherent $\mc{O}_Y$-modules.  Suppose that $M$ is the sheaf
defined in the theorem.  Let 
\begin{eqnarray}
\mathcal{M} = \sheafHom_{\mathbb{F}_p} \left( M , \mc{O}_{Y_\textnormal{\'et}} \right).
\end{eqnarray}
Then $\mathcal{M}$ has the structure of an $\mc{O}_Y$-module.  Additionally,
the $p$th-power map on $\mc{O}_{Y_\textnormal{\'et}}$ determines
a Frobenius-linear endomorphism $F \colon \mc{M} \to \mc{M}$.  
From the data of $\mc{M}$ together with this endomorphism, one
can recover the original sheaf $M$.  This association is part
of the characteristic-$p$ ``Riemann-Hilbert'' correspondence
of M.~Emerton and M.~Kisin (\cite{emertonandkisin}).  Following their terminology, we
call $\mc{M}$ an ``$\mc{O}_{F, Y}$-module.''

The characteristic-$p$ Riemann-Hilbert correspondence is
developed in full generality in \cite{emertonandkisin}.  Since
we prefer to avoid the language of derived categories, we will
not make direct use of the results from that paper.  We 
construct a miniature version of the correspondence
which applies to $\mathbb{F}_p$-sheaves on $k$-curves.  
Our version includes a localization functor.  The key
results are Theorem~\ref{localRH}
and Propositions~\ref{stalkiso1}, \ref{stalkiso2},
\ref{globaldoubledual1}, and \ref{globaldoubledual2}.

A key ingredient for the construction of Theorem~\ref{mainthmquoted}--the ingredient,
in fact, that allows the extension beyond
Pink's original result--is the notion of a ``root'' of
an $\mc{O}_{F, Y}$-module.
This notion is due to G.~Lyubeznik (see \cite{lyubeznikfmodules}).
A ``root'' for $\mc{M}$ is
a special type of coherent generating submodule.
If $\mc{M}_0 \subseteq \mc{M}$ is a root, then the
images of $\mc{M}_0$ under repeated applications of the map
$F$ determine an ascending filtration for the sheaf $\mc{M}$.
(See Definition~\ref{rootdef} and Proposition~\ref{rootfiltration}).
In this paper we develop the properties of roots
in parallel with the construction of the Riemann-Hilbert correspondence.

One interesting result that arises 
is Theorem~\ref{minrootthm}, which asserts the existence
of canonical minimal roots for $\mc{O}_{F, X}$-modules in dimension one.
(This result is critical for 
Theorem~\ref{mainthmquoted}.  The local term
$\mathfrak{C} \left( M_{(y)} \right)$ is based on a measurement
of the minimal root of $\mc{M}$.)  I am pleased
to point out that a much stronger version of this result
has been proven, independently, by M.~Blickle (\cite{blickleminroot}).
Blickle proved the existence of canonical minimal roots over
any $F$-finite regular ring.
In particular, this means that minimal roots exist on
smooth $k$-schemes of arbitrary dimension.
(This naturally suggests an extension of the definition
of $\mathfrak{C}$!)

Section~\ref{ofxmodulesection} in this paper reviews terminology for
$\mc{O}_{F, X}$-modules and establishes some basic results.
Sections~\ref{localsection} and \ref{RHcurvesection} develop
the one-dimensional Riemann-Hilbert correspondence.
(Section~\ref{localsection} contains local results, and
Section~\ref{RHcurvesection} contains global results.)
Theorem~\ref{mainthmquoted} appears in subsection~\ref{cohomologysubsection}.
(The proof is similar to the one used in \cite{eulerpink}.)
The paper closes with three examples involving sheaves
on the projective line.

\subsection{Further directions}

A natural goal is to determine conditions under which
formula~(\ref{mainthmquotedformula}) yields an equality.  Empirical
evidence suggests that equality occurs ``generically,'' and it would
be interesting to make that statement more precise.
Another goal is to gain a better understanding of the local invariant
$\mathfrak{C} \left( M_{(y)} \right)$.
The invariant is defined in this paper in terms $\mc{O}_{F, Y}$-modules,
but it should be possible to understand it directly in terms of the
localization $M_{(y)}$ (which is a sheaf on $\Spec \mc{O}_{Y_{\textnormal{\'et}},y}$).
The localization $M_{(y)}$ is essentially an equicharacteristic Galois representation.

Another direction has to do with $p$-adic cohomology.  Theorem~\ref{mainthmquoted}
is an assertion about $p$-torsion sheaves, but without much difficulty
it can be converted into a statement about \'etale $\mathbb{Q}_p$-sheaves.
The $\mathbb{Q}_p$-version of the theorem might have interesting connections
with other known results on $p$-adic cohomology.  (Consider
for example Theorem~4.3.1 in \cite{kedlaya}, which is a Grothendieck-Ogg-Shafarevich
formula for rigid cohomology.)

\subsection{Acknowledgements} This paper is an abridged version
of my dissertation at UC-Berkeley.  I want to gratefully acknowledge the
help of my mentors, Arthur Ogus, Martin Olsson, and Brian Conrad, who
have had an extensive influence on the shape of the material here.
Special thanks go to Arthur Ogus (my thesis advisor), who was my audience when I was
working on the main result.  Also, I want to thank some other
colleagues for conversations about the material (some brief
but enlightening): Manuel Blickle, Igor Dolgachev,
Kiran Kedlaya, Gennady Lyubeznik,
Mark Kisin, Jacob Lurie, Brian Osserman,
Richard Pink, Bjorn Poonen, Karen Smith,
and Nicolas Stalder.

\subsection{Notation and conventions}

Throughout this paper, $p$ denotes a prime, $r$ denotes a
positive integer, and $k$ denotes an algebraically closed field
of characteristic $p$.
All sheaves are assumed to be sheaves on an \'etale site.
Thus, if $X$ is a $k$-scheme, $\mc{O}_X$ denotes the
\'etale coordinate sheaf on $X$.  If $x$ is a $k$-point of $X$,
then $\mc{O}_{X, x}$ denotes the (\'etale) stalk of $\mc{O}_X$
at $x$.

If $R$ is a ring, let $\lmod ( R )$ (or simply $\mod ( R)$,
if $R$ is commutative) denote the category of left
$R$-modules.  If $X$ is a scheme and $\mathcal{R}$ is
a sheaf of rings on $X$, let $\textnormal{(}\textbf{L}\textnormal{)}\textbf{Mod}
( X, \mathcal{R} )$ denote the category of (left) $\mathcal{R}$-modules
on $X$.

If $S$ is a $k$-algebra, let $F_S \colon S \to S$ denote
the Frobenius map.  If $Z$ is a $k$-scheme, let $F_Z \colon 
Z \to Z$ denote the Frobenius morphism.

All schemes are assumed to be noetherian and separated.

\section{$\mc{O}_{F^r,X}$-modules}

\label{ofxmodulesection}

Let $X$ be a $k$-scheme.
This section is concerned with quasi-coherent $\mc{O}_X$-modules
that have Frobenius-linear endomorphisms.  For the study of these modules it is convenient to
introduce the sheaf $\mc{O}_{F^r, X}$. This is a sheaf of noncommutative
rings in which the multiplication rule is determined by the $r$th Frobenius
map on $\mc{O}_X$.  Notation and terminology in this section
are borrowed from \cite{emertonandkisin}.

\subsection{The category of left $\mc{O}_{F^r, X}$-modules}

Let $S$ be a $k$-algebra.
Then $S[F^r]$ is the twisted
polynomial ring determined by the $r$th Frobenius endomorphism,
\begin{eqnarray}
(F_S)^r \colon S \to S.
\end{eqnarray}
Elements of the ring $S[F^r]$ may be expressed as finite sums of the
form
\begin{eqnarray}
\sum_{i \geq 0} s_i F^{ri},
\end{eqnarray}
and multiplication is determined by the
rule $F^r s = s^{(p^r)} F^r$.  Similarly, let $X$ be a $k$-scheme.
Then $\mc{O}_{F^r, X}$ denotes the sheaf of twisted
polynomial rings determined by $F_X^r \colon X \to X$.  
If $\Spec T \subseteq X$ is any affine open inside of $X$,
then the sections of $\mc{O}_{F^r, X}$ over $\Spec T$ form the
ring $T[F^r]$.

This notation provides a convenient way to express
Frobenius-linear endomorphisms.  Let $M$ be an $S$-module.
Then a left $S[F^r]$-module structure of $M$ is uniquely specified
by an additive endomorphism $\phi \colon M \to M$ satisfying
$\phi ( sm ) = s^{(p^r)} \phi (m)$ for all $s \in S$, $m \in M$.  (The
map $\phi$ expresses the left action of $F^r$ on $M$.)  Thus
a left $S[F^r]$-module is simply an $S$-module with a Frobenius-linear
endomorphism.

If $\mc{M}$ is a left $\mc{O}_{F^r, X}$-module, then the
sheaf endomorphism $\mc{M} \to \mc{M}$ determined by the action 
of $F^r$ determines a morphism
\begin{eqnarray}
F^{r*}_X \mc{M} \to \mc{M}
\end{eqnarray}
which is $\mc{O}_X$-linear.  We refer to this homomorphism
as the structural morphism of $\mc{M}$.

\begin{definition}
A \textnormal{unit $\mc{O}_{F^r, X}$-module} is a left $\mc{O}_{F^r, X}$-module
which is quasi-coherent (as an $\mc{O}_X$-module) and 
whose structural morphism is an isomorphism.
\end{definition}

The term ``unit'' may be similarly applied to modules over rings.
A left $S[F^r]$-module is unit if its structural homomorphism
$F^{r*}_S M \to M$ is an isomorphism.  Let $\lmod^u (S[F^r])$
be the full subcategory of unit $S[F^r]$-modules in
$\lmod ( S[F^r] )$.  Additionally, let $\lmod^{fu} (S[F^r])$
be the subcategory of $\lmod^u ( S[F^r] )$ consisting
of those objects which are finitely-generated left $S[F^r]$-modules.
(We shall call these simply ``finitely-generated unit $S[F^r]$-modules.'')

The categories $\lmod^u ( X, \mc{O}_{F^r, X} )$ and $\lmod^{fu}
( X, \mc{O}_{F^r, X} )$ are defined similarly.  The category
$\lmod^u ( X, \mc{O}_{F^r, X} )$ consists of the unit $\mc{O}_{F^r, X}$-modules,
and the category $\lmod^{fu} ( X, \mc{O}_{F^r, X} )$ consists of those
unit $\mc{O}_{F^r, X}$-modules which are finitely-generated once restricted
to any affine open subset of $X$.  Using terminology from
\cite{emertonandkisin}, we refer to objects from $\lmod^{fu} ( X, 
\mc{O}_{F^r, X} )$ as {\it locally finitely-generated unit (lfgu)
$\mc{O}_{F^r, X}$-modules}.  (Note that ``finitely-generated''
refers to the left $\mc{O}_{F^r, X}$-module structure of 
the sheaf in question, not to its $\mc{O}_X$-module structure.)

Let $f \colon X \to Y$ be a morphism of schemes over $k$, and
let $\mc{M}$ be a unit $\mc{O}_{F^r, Y}$-module.
The $\mc{O}_X$-module pullback $f^* \mc{M}$ of $\mc{M}$ is given by
\begin{eqnarray}
f^* \mc{M} = \mc{O}_X \otimes_{f^{-1} \mc{O}_Y} f^{-1} \mc{M}.
\end{eqnarray}
There is a natural left $\mc{O}_{F^r, X}$-module structure
on $f^* \mc{M}$ which is expressed by the rule
\begin{eqnarray}
F^r ( f \otimes m ) = f^{(p^r)} \otimes F^r ( m ).
\end{eqnarray}

\begin{proposition}
Let $f \colon X \to Y$ be a morphism of schemes over $k$, and let
$\mc{M}$ be a unit $\mc{O}_{F^r, Y}$-module.  Then the pullback
$f^* \mc{M}$ is a unit $\mc{O}_{F^r, X}$-module.  If $\mc{M}$
is locally finitely-generated, then $f^* \mc{M}$ is locally
finitely-generated.
\end{proposition}

\begin{proof}
The structural morphism of $f^* \mc{M}$ is the composite of
two morphisms,
\begin{eqnarray}
F_X^{r*} f^* \mc{M} \to f^* F_Y^{r*} \mc{M} \to f^* \mc{M},
\end{eqnarray}
where the first arises from the commutative diagram
\begin{eqnarray}
\xymatrix{
X \ar[r]^{F^r_X} \ar[d]^f & X \ar[d]^f \\
Y \ar[r]^{F^r_Y} & Y, }
\end{eqnarray}
and the second is determined by the structural morphism
for $\mc{M}$.  Both maps are isomorphisms.  Therefore
$f^* \mc{M}$ is a unit $\mc{O}_{F^r, X}$-module.

Now suppose that $\mc{M}$ is locally finitely-generated.
The condition that $f^* \mc{M}$ is lfgu needs only to be checked
locally.  Choose any closed point $x$ in $X$.  Let $V \subset Y$
be an affine open subscheme which contains $f(x)$.  Let
$\{ m_1, \ldots, m_r \} \subseteq \mc{M} ( V )$ be a set
which generates $\mc{M}_{\mid V}$ as
a left $\mc{O}_{F^r, V}$-module.  The pullbacks of 
$\{ m_1, \ldots, m_r \}$ to $f^{-1} (V)$ generate
$(f^* \mc{M})_{\mid f^{-1} (V)}$ as a left 
$\mc{O}_{F^r, f^{-1} (V)}$-module.
Thus $f^* \mc{M}$ is
finitely-generated on an open neighborhood of $x$. 
\end{proof}

\subsection{Roots of lfgu $\mc{O}_{F^r, X}$-modules}

While sheaves in the category $\lmod^{fu} ( X, \mc{O}_{F^r, X} )$ are not
necessarily coherent, they have special coherent 
subsheaves which capture their structure.  The concept
of a root is due to Lyubeznik (see \cite{lyubeznikfmodules}).

\begin{definition}
\label{rootdef}
Let $X$ be a $k$-scheme, and let $\mc{M}$ be
a unit $\mc{O}_{F^r, X}$-module.  An $\mc{O}_X$-submodule
$\mc{M}' \subseteq \mc{M}$ is a \textnormal{root}
if
\begin{enumerate}
\item the $\mc{O}_X$-module $\mc{M}'$ is coherent,

\item the $\mc{O}_X$-submodule of $\mc{M}$ generated
by $F^r \left( \mc{M}' \right)$ contains $\mc{M}'$, and

\item as a left $\mc{O}_{F^r, X}$-module, $\mc{M}$ is
generated by $\mc{M}'$.
\end{enumerate}
\end{definition}

If a unit $\mc{O}_{F^r, X}$-module has a root, then
it is locally finitely-generated.  (This is easily deduced
from properties $1$ and $3$ above.)  The next proposition asserts
that the converse is also true.

\begin{proposition}
Let $Y$ be a smooth $k$-scheme, and let $\mc{M}$ 
be a unit $\mc{O}_{F^r, Y}$-module.  If $\mc{M}$
is locally finitely-generated, then it has a root.
\end{proposition}

\begin{proof}
This is Theorem~6.1.3 from \cite{emertonandkisin}.
\end{proof}

The next two propositions will be useful in later parts
of this paper.  The first asserts that a root
of a unit $\mc{O}_{F^r, X}$-module determines a filtration
of the module by coherent subsheaves.  The second proposition
asserts that this filtration collapses on an open dense
subset of the scheme $X$.

\begin{proposition}
\label{rootfiltration}
Let $Y$ be a smooth $k$-scheme.  Let $\mc{M}$ be an
lfgu $\mc{O}_{F^r, Y}$-module, and let $\mc{M}_0 \subseteq
\mc{M}$ be a root for $\mc{M}$.  For each $n \geq 1$, let
$\mc{M}_n$ be the $\mc{O}_X$-submodule of $\mc{M}$ generated by
$F^{rn} \left( \mc{M}_0 \right)$.  The sequence 
$\left( \mc{M}_n \right)$ is an ascending filtration of $\mc{M}$.
For each $n \geq 0$, the structural morphism of $\mc{M}$ determines
a map
\begin{eqnarray}
\label{filtrationisomorphism}
F^{rn*}_X \mc{M}_0 \to \mc{M}_n
\end{eqnarray}
which is an isomorphism.
\end{proposition}

\begin{proof}
We prove the last assertion first.  Since $X$ is smooth, the 
Frobenius morphism $F^{rn}_X \colon X \to X$ is flat and finite.
Therefore the inclusion
\begin{eqnarray}
\mc{M}_0 \hookrightarrow \mc{M}
\end{eqnarray}
induces an injection
\begin{eqnarray}
F^{rn*}_X \mc{M}_0 \hookrightarrow F^{rn*}_X \mc{M}
\end{eqnarray}
Composing this injection with the $n$th power of the structural
morphism of $\mc{M}$ yields an injection
\begin{eqnarray}
F^{rn*}_X \mc{M}_0 \hookrightarrow \mc{M}
\end{eqnarray}
whose image (by definition) is $\mc{M}_n$.
Thus we obtain the desired isomorphism.

The assertion that $\left( \mc{M}_n \right)$ is an ascending
filtration follows from the observation that $\mc{M}_{n+1} 
\subseteq \mc{M}$ is the sub-$\mc{O}_X$-module generated
by $F^{rn} ( \mc{M}_1 )$.  Since $\mc{M}_1$ contains
$\mc{M}_0$ (by property $2$ of Definition~$\ref{rootdef}$), 
$\mc{M}_{n+1}$ contains $\mc{M}_n$.  Property $3$
of Definition~$\ref{rootdef}$ implies that the union
of the submodules $\mc{M}_n$ is $\mc{M}$. 
\end{proof}

\begin{proposition}
\label{densecoherent}
Let $Y$ be a smooth $k$-scheme.  Let $\mc{M}$ be
an lfgu $\mc{O}_{F^r, Y}$-module, and let $\mc{M}_0
\subseteq \mc{M}$ be a root for $\mc{M}$.  Then there
exists an open dense subset $U \subseteq Y$ such that
$\mc{M}_{0 \mid U} = \mc{M}_{\mid U}$.
\end{proposition}

\begin{proof}
Clearly we may assume that $Y$ is irreducible (and
therefore integral).  Let $\left( \mc{M}_n \right)$
be the filtration from Proposition~\ref{rootfiltration}.
The isomorphism
\begin{eqnarray}
F^{r*}_Y \mc{M}_0 \to \mc{M}_1
\end{eqnarray}
implies that
the generic rank of $\mc{M}_1$ is the same as the generic
rank of $\mc{M}_0$.  Therefore $\mc{M}_1 / \mc{M}_0$ is
supported at a proper closed subset of $Y$.  Let $U \subseteq Y$
be the complement of this closed subset.  The sheaf 
$\mc{M}_{0 \mid U}$ is stabilized by $F^r$.  Since
$\mc{M}_{0 \mid U}$ generates $\mc{M}_{\mid U}$,
$\mc{M}_{0 \mid U}$ must coincide with $\mc{M}_{\mid U}$.
\end{proof}

Lastly, we note that the reasoning used in the last two proofs
also proves an important fact: any finitely-generated unit
$\mc{O}_{F^r, X}$-module over a field must be finite-dimensional.

\begin{proposition}
\label{findim}
Let $L$ be a field which contains $k$.  Then a unit $L[F^r]$-module
is finitely-generated if and only if it is finite-dimensional over $L$.
\end{proposition}

\begin{proof}
Let $V$ be a finitely-generated unit $L[F^r]$-module.  Choose
an algebraic closure $\overline{L}$ for $L$.  The pullback
$\overline{L} \otimes_L V$ is a finitely-generated unit $\overline{L}[F^r]$-module,
which must have a root (by Theorem 6.1.3 from \cite{emertonandkisin}).
Let $\overline{V}_0 \subseteq \overline{L} \otimes_L V$ be a root.
As in the proof of Proposition~\ref{rootfiltration}, this root
determines a filtration of $\overline{L} \otimes V$,
\begin{eqnarray}
\overline{V}_0 \subseteq \overline{V}_1 \subseteq \overline{V}_2
\subseteq \ldots ,
\end{eqnarray}
in which adjacent terms have isomorphisms $F_{\overline{L}}^{r*} \overline{V}_n \cong
\overline{V}_{n+1}$.  Each term in this filtration must have the same (finite)
dimension.  Therefore $\overline{L} \otimes V$ is finite-dimensional
over $\overline{L}$, and $V$ is finite-dimensional over $L$.

The converse is immediate.
\end{proof}

\section{A local analysis of $\mc{O}_{F^r, X}$-modules in dimension one}

\label{localsection}

Throughout this section,
let $A$ be a Henselization of the local ring $k[t]_{(t)}$.  (For
example, one can let $A$ be the ring of elements of $k[[t]]$ 
that are algebraic over $k(t)$.)  Additionally, let $K$ 
denote the fraction field of $A$.
Note that if $X$ is any smooth $k$-curve,
then the stalk of its \'etale coordinate sheaf at any
closed point is isomorphic to $A$.  Thus unit $\mc{O}_{F^r, X}$-modules
localize to unit $A[F^r]$-modules.

This section is concerned with finitely-generated unit $A[F^r]$-modules.  We 
are primarily interested in those which
are torsion-free as $A$-modules.
The goals of this section are (1) to establish the local Riemann-Hilbert
correspondence (Theorem~\ref{localRH}), and (2) to define the ``minimal root index'' of
a unit $A[F^r]$-module.

The following algebraic result is a starting point.

\begin{proposition}
\label{sepclosedfrobenius}
Let $L$ be a separably closed field of characteristic $p$.  Let
$H$ be a unit $L[F^r]$-module which is finite-dimensional over
$L$.  Then the set $H^{(F^r)}$ of $F^r$-invariant elements of $H$
spans $H$.  This set forms an $\mathbb{F}_{p^r}$-vector space.
The map
\begin{eqnarray}
H^{(F^r)} \otimes_{\mathbb{F}_{p^r}} L \to H
\end{eqnarray}
is an isomorphism.
\end{proposition}

\begin{proof}
This is a reformulation of Proposition~1.1 from
\cite{katzsga}.
\end{proof}

Note that this proposition implies that every 
unit $L[F^r]$-module which is finite-dimensional
over $L$ has an $F^r$-invariant basis.

\subsection{Trivializations of unit $A[F^r]$-modules}

Proposition~\ref{sepclosedfrobenius}
implies trivializations for unit $A[F^r]$-modules
under certain assumptions.
We state several assertions here
for later use.

\begin{proposition}
\label{trivfieldext}
Let $V$ be a finitely-generated unit $K[F^r]$-module,
and let $n$ be the dimension of $V$ as a $K$-vector space.
Let $K^{sep}$ be a separable closure of the field $K$.
Then there exists an isomorphism of left $K^{sep}[F^r]$
modules,
\begin{eqnarray}
\label{trivfieldextiso}
K^{sep} \otimes_K V \cong \left( K^{sep} \right)^{\oplus n}.
\end{eqnarray}
(The left $K^{sep}[F^r]$-module structure for the
vector space on the right is given by the Frobenius
map $F^r_{K^{sep}}$.)
\end{proposition}

\begin{proof}
The module $K^{sep} \otimes_K V$ is a unit $K^{sep}[F^r]$-module
which is finite-dimensional over $K^{sep}$.  By Proposition~\ref{sepclosedfrobenius}
it has $K^{sep}$-basis which is $F^r$-invariant.  This basis
determines the isomorphism.
\end{proof}

\begin{corollary}
\label{trivfieldextcor}
For some finite separable field extension $K' / K$,
there exists an isomorphism of left $K'[F^r]$-modules
\begin{eqnarray}
K' \otimes_K V \cong \left( K' \right)^{\oplus n}.
\end{eqnarray}
\end{corollary}

\begin{proof}
Let $B \subseteq K^{sep} \otimes_K V$
be the basis which determines isomorphism (\ref{trivfieldextiso}).
Simply choose $K' \subseteq K^{sep}$ large enough that
$K' \otimes_K V \subseteq K^{sep} \otimes_K V$ contains $B$.
\end{proof}

\begin{proposition}
\label{trivsubmod}
Let $W$ be a torsion-free unit $A[F^r]$-module such $K \otimes_A W$
is isomorphic to $K^{\oplus n}$ as a left $A[F^r]$-module.  Then
there exists an isomorphism
\begin{eqnarray}
\label{trivsubmodiso}
W \cong K^{\oplus m} \oplus A^{\oplus n-m}
\end{eqnarray}
with $0 \leq m \leq n$.
\end{proposition}

\begin{proof}
Let us consider $W$ as a submodule of $K \otimes W$.
Let $\left\{ w_1, \ldots, w_n \right\} \subseteq K \otimes W$ be
the basis which determines the isomorphism to $K^{\oplus n}$.  
This basis is $F^r$-invariant.
Note that $t^n w_1 \in W$ for sufficiently large $n$.  I claim
that in fact $n = 0$ is sufficient.  For, suppose not: then
$t^N w_1 \in W$ and $t^{N-1} w_1 \notin W$ for some $N > 0$.
But in this case there can be no way to express $t^N w_1$
as an $A$-linear combination of elements from $F^r ( W )$.
So $W$ could not be a unit $A[F^r]$-module.  Thus $w_1$
(and likewise every other element from $\{ w_i \}$) must
be contained in $W$.

Let $V \subseteq W$ be the $\mathbb{F}_{p^r}$-vector space
spanned by $\left\{ w_1, \ldots, w_n \right\}$.  Choose
an $\mb{F}_{p^r}$-basis $\left\{ w'_1, \ldots, w'_m \right\}$ for 
the subspace $V \cap tW$.  Each element $w'_i$ satisfies
$t^{-1} w'_i \in W$.  Since $W$ is closed under action by Frobenius,
this implies $t^{-p^{rN}} w'_i \in W$ for any $N > 0$.
Extend $\left\{ w'_1, \ldots, w'_m \right\}$ to a basis
$\left\{ w'_1, \ldots, w'_n \right\}$ for the entire
space $V$.  This basis determines isomorphism (\ref{trivsubmodiso}).
\end{proof}

\begin{proposition}
\label{trivringext}
Let $W$ be an object from $\lmod^{fu} ( A[F^r] )$ which
is a torsion-free $A$-module.  Then there exists a finite
integral extension
\begin{eqnarray}
\xymatrix{
A \ar[r] \ar[d] & A' \ar[d] \\
K \ar[r] & K',
}
\end{eqnarray}
and a left $A'[F^r]$-module isomorpism
\begin{eqnarray}
A' \otimes_A W \cong \left( K' \right)^{\oplus m} \oplus
\left( A' \right)^{\oplus n-m}
\end{eqnarray}
with $0 \leq m \leq n$.
\end{proposition}

\begin{proof}
Choose a field extension $K' / K$ according to
Corollary~\ref{trivfieldextcor} so that $K' \otimes_A W$
is isomorphic to $\left( K' \right)^{\oplus n}$ for 
some $n$.  Let $A'$ be the integral closure of $A$
inside of $K'$.  Note that the Heneselian DVRs
$A$ and $A'$ are in
fact isomorphic.  Thus Proposition~\ref{trivsubmod}
can be translated into a statement about
modules over $A'$:
\begin{itemize}
\item If $W'$ is any torsion-free
unit $A'[F^r]$-module such that $K' \otimes W'
\cong \left( K' \right)^{\oplus n}$, then
$W' \cong {K'}^{\oplus m} \oplus {A'}^{\oplus n-m}$
for some $m$.
\end{itemize}
The proposition follows once we let $W' = A' \otimes_A W$.
\end{proof}

\begin{proposition}
\label{Amodtriv}
Let $W$ be a free finite-rank $A$-module which has a unit
$A[F^r]$-module structure.  Then $W$ is isomorphic
as a left $A[F^r]$-module to $A^{\oplus n}$ for some $n$.
\end{proposition}

\begin{proof}
Consider the quotient $W / tW$, which is a finite-dimensional
unit $k[F^r]$-module.  By Proposition~\ref{sepclosedfrobenius},
this vector space has an $F^r$-invariant basis.  To prove
Proposition~\ref{Amodtriv}, it suffices to show that that this
basis can be lifted to an $F^r$-invariant $A$-module basis
for $W$.  This is accomplished by the following lemma:

\begin{lemma}
Any $F^r$-invariant element of $W / tW$ can be uniquely lifted
to an $F^r$-invariant element of $W$.
\end{lemma}

\begin{proof}
The lemma may be formulated in terms of commutative algebra.
Let $\left\{ w_1, \ldots, w_n \right\}$ be any $A$-module basis
for $W$.  Since
$W$ is a unit $A[F^r]$-module, the set $\left\{ F^r (w_1) , 
\ldots , F^r ( w_n ) \right\}$ is another basis, and
there exists an invertible $A$-matrix $\left( c_{ij} \right)$
such that
\begin{eqnarray}
w_i = \sum_{j=1}^n c_{ij} F^r ( w_j ).
\end{eqnarray}
An element
\begin{eqnarray}
\sum_{k=1}^n a_k w_k \in W \hskip0.3in (a_k \in A)
\end{eqnarray}
is $F^r$-invariant if and only if
\begin{eqnarray}
\label{freeeqns}
\sum_{k=1}^n a_k^{p^r} F^r (w_k ) = 
\sum_{k=1}^n a_k w_k  = 
\sum_{k=1}^n a_k \sum_{j=1}^n c_{kj} F^r ( w_j ),
\end{eqnarray}
or equivalently,
\begin{eqnarray}
a_k^{p^r} = \sum_{\ell = 1}^n a_\ell c_{\ell k}
\end{eqnarray}
for each $k = 1, 2, \ldots, n$.
Let
\begin{eqnarray}
R = A[X_1, \ldots, X_n]/ \left( \left\{ X_k^{p^r} -
\sum_{\ell=1}^n X_\ell c_{\ell k} \right\}_{k=1}^n \right).
\end{eqnarray}
Then $F^r$-invariant elements of $W$ may be specified
by $A$-homomorphisms from $R$ into $A$, while $F^r$-invariant
elements of $W / tW$ may be specified by $A$-homomorphisms
from $R$ into $k$.  The claim made in the lemma, then,
is equivalent to the assertion that every element
of $\Hom_A ( R, k )$ can be lifted to an element
of $\Hom_A ( R, A )$.

This assertion becomes evident once we understand
the structure of $R$.  The extension
$A \to R$ is finite, flat, and unramified (as the
reader may check), and therefore \'etale.  Since $A$
is a Henselian local ring, $R$ is simply a finite direct
sum of copies of $A$.
\end{proof}

Now we may complete the proof of Proposition~\ref{Amodtriv}.
Choose an $F^r$-invariant $k$-basis for $W / tW$.  There
is a unique $F^r$-invariant lifting of this set to $W$,
and by Nakayama's lemma this lifting is an $A$-module basis.
\end{proof}

\subsection{The structure of a unit $A[F^r]$-module}

\label{unitstruct}

If $W$ is a torsion-free unit $A[F^r]$-module, let
\begin{eqnarray}
W^{vec} = \bigcap_{n=0}^\infty t^n W \subseteq W.
\end{eqnarray}
The set $W^{vec}$ is the largest $K$-vector
space contained inside of $W$.  (This definition 
extends naturally to modules over finite integral
extensions $A \hookrightarrow A'$ as well.)
It is easily checked that $W^{vec}$ is stabilized 
by the action of $F^r$.  Thus there is an exact sequence
of left $A[F^r]$-modules,
\begin{eqnarray}
0 \to W^{vec} \to W \to W / W^{vec} \to 0.
\end{eqnarray}
This exact sequence will be the basis for the
proof of the local Riemann-Hilbert correspondence.

The reader may verify
the following elementary assertions for any
torsion-free unit $A[F^r]$-module $W$:
\begin{enumerate}
\item The quotient $W / W^{vec}$ inherits the
structure of a unit $A[F^r]$-module, and $W^{vec}$
inherits the structure of a unit $K[F^r]$-module.

\item If $V$ is any $K$-vector subspace of $W^{vec}$,
then $\left( W / V \right)^{vec} = W^{vec} / V$.

\item If $A \hookrightarrow A'$ is any 
finite integral extension, then $\left( A' \otimes_A
W \right)^{vec} = A' \otimes_A W^{vec}$. 
\end{enumerate}

\begin{proposition}
\label{trivquot}
Let $W$ be an object from $\lmod^{fu} ( A[F^r] )$ which
is a torsion-free $A$-module.  Then $W^{vec}$ is a 
finite-dimensional $K$-vector space, and $W / W^{vec}$
is isomorphic as a left $A[F^r]$-module to $A^{\oplus n}$
for some $n$.
\end{proposition}

\begin{proof}
The $K$-vector space $W \otimes_A K$ is finite-dimensional
(by Proposition~\ref{findim}), and $W^{vec}$ is a $K$-subspace
of $W \otimes_A K$.  The first assertion follows.
Let $U = W / W^{vec}$.  Then $U$ is a torsion-free unit $A[F^r]$-module
such that $U^{vec} = \left\{ 0 \right\}$.  
By Proposition~\ref{Amodtriv}, the proof will be completed if
we can show that $U$ is finitely-generated as an $A$-module.

Choose a finite integral extension $(A', K')$ of $(A, K)$
according to Proposition~\ref{trivringext} so that
\begin{eqnarray}
\label{trivringextiso2}
A' \otimes_A U \cong \left( K' \right)^{\oplus m} \oplus
\left( A' \right)^{\oplus n-m}.
\end{eqnarray}
Since $\left( A' \otimes_A U \right)^{vec} = A' \otimes_A U^{vec}
= \left\{ 0 \right\}$, we must have $m=0$ above.  Isomorphism (\ref{trivringextiso2})
makes $U$ isomorphic to an $A$-submodule of $\left( A' \right)^{\oplus n}$.
Therefore $U$ is a finitely-generated $A$-module.
\end{proof}

The following corollary includes a converse to Proposition~\ref{trivquot}.
The proof is easy and is left to the reader.

\begin{corollary}
\label{fgustruct}
Let $Y$ be a left $A[F^r]$-module.  Then the following conditions
are equivalent:
\begin{enumerate}
\item The module $Y$ is a finitely-generated
unit $A[F^r]$-module that has no $A$-torsion.
\item There exists an exact sequence of left $A[F^r]$-modules,
\begin{eqnarray}
\label{unitexactseq}
0 \to Y' \to Y \to A^{\oplus n} \to 0,
\end{eqnarray}
where $Y'$ is a finitely-generated unit $K[F^r]$-module. $\Box$
\end{enumerate}
\end{corollary}

\subsection{The local Riemann-Hilbert correspondence}

It is helpful at this point to introduce some geometric
notation.  Let $Z = \Spec A$.  Let $s$ be the closed
point of $Z$ (which has residue field $k$),
and let $\eta$ be the generic point of $Z$ (which has
residue field $K$).
Let
\begin{eqnarray}
\overline{\eta} \colon \Spec K^{sep} \to Z
\end{eqnarray}
be a geometric point at $\eta$.  (Here $K^{sep}$ denotes
a separable closure of $K$.)

Let $V$ be a constructible sheaf of $\mb{F}_{p^r}$-vector spaces
on the scheme $\{ \eta \} \subseteq Z$.  Since $V$ is constructible,
its stalk $M_{\overline{\eta}}$ is finite.  Let 
\begin{eqnarray}
\mc{V}' = \sheafHom_{\mb{F}_{p^r}} \left( V , \mc{O}_{\{ \eta \}} \right).
\end{eqnarray}
Galois descent implies that $\mc{V}$' is a quasi-coherent
$\mc{O}_{ \{ \eta \} }$-module.  Moreover, the $r$th Frobenius endomorphism
of $\mc{O}_{ \{ \eta \} }$ determines a left $\mc{O}_{F^r, \{ \eta \}}$-structure
on $\mc{V}'$ which makes $\mc{V}$' a finitely-generated unit $\mc{O}_{ F^r,
\{ \eta \} }$-module.

At the same time, if $\mc{V}$ is a finitely-generated unit $\mc{O}_{ F^r, 
\{ \eta \} }$-module, then
\begin{eqnarray}
V' = \sheafHom_{\mc{O}_{F^r, \{ \eta \}}} \left( \mc{V} , \mc{O}_{ \{ \eta
\} } \right)
\end{eqnarray}
is a sheaf of $\mb{F}_{p^r}$-vector spaces on $\{ \eta \}$.  The stalk
of $V'$ at $\overline{\eta}$ is
\begin{eqnarray}
V'_{\overline{\eta}} =\Hom_{K^{sep}} \left( \mc{V}_{ \overline{\eta} } ,
K^{sep}
\right),
\end{eqnarray}
which is made isomorphic to $\mb{F}_{p^r}^n$ for some $n$ by 
Proposition~\ref{trivfieldext}.  Thus $V'$ is a constructible
$\mb{F}_{p^r}$-\'etale sheaf.  There is a natural double-dual homomorphism
\begin{eqnarray}
\mc{V} \to \sheafHom_{\mb{F}_{p^r}} \left( \sheafHom_{\mc{O}_{F^r, \{ 
\eta \}}} \left( \mc{V} , \mc{O}_{ \{ \eta \} } \right) , 
\mc{O}_{ \{ \eta \} } \right)
\end{eqnarray}
which is easily seen to be an isomorphism by computing stalks
at $\overline{\eta}$.  Likewise, the double-dual homomorphism
\begin{eqnarray}
V \to \sheafHom_{\mc{O}_{F^r, \{ \eta \}}} \left(
\sheafHom_{\mb{F}_{p^r}} \left( V, \mc{O}_{ \{ \eta \} }
\right) , \mc{O}_{ \{ \eta \} } \right)
\end{eqnarray}
is an isomorphism.

Let $\mod^c \left( \{ \eta \} , \mb{F}_{p^r} \right)$ be the
full subcategory of constructible sheaves in $\mod \left( 
\{ \eta \} , \mb{F}_{p^r} \right)$.  The functors
$\sheafHom_{\mb{F}_{p^r}} \left( \cdot , \mc{O}_{ \{ \eta
\} } \right)$ and $\sheafHom_{\mc{O}_{F^r, \{ \eta \}}}
\left( \cdot , \mc{O}_{ \{ \eta \} } \right)$ determine an
equivalence of categories between $\lmod^{fu} \left( \{
\eta \} , \mc{O}_{F^r, \{ \eta \}} \right)$ and
$\mod^c \left( \{ \eta \} , \mb{F}_{p^r} \right)$.  The 
local Riemann-Hilbert correspondence simply extends this 
equivalence to the scheme $Z$.

\begin{theorem}
\label{localRH}
Let $M$ be a constructible $\mb{F}_{p^r}$-\'etale sheaf on $Z$ whose sections
all have open support.  Then the sheaf
\begin{eqnarray}
\mc{M}' = \sheafHom_{\mb{F}_{p^r}} \left( M , \mc{O}_Z \right)
\end{eqnarray}
is a finitely-generated unit $\mc{O}_{F^r, Z}$-module.  The double-dual
homomorphism
\begin{eqnarray}
M \to \sheafHom_{\mc{O}_{F^r, Z}} \left( \mc{M}' , \mc{O}_Z \right)
\end{eqnarray}
is an isomorphism.

Let $\mc{M}$ be a sheaf from $\lmod^{fu} \left( Z, \mc{O}_{F^r, Z} \right)$
which has a torsion-free $\mc{O}_Z$-module structure.
Then the sheaf
\begin{eqnarray}
M' = \sheafHom_{\mc{O}_{F^r, Z}} \left( \mc{M} , \mc{O}_Z \right)
\end{eqnarray}
is a constructible sheaf of $\mb{F}_{p^r}$-vector spaces.  The double-dual
homomorphism
\begin{eqnarray}
\mc{M} \to \sheafHom_{\mc{O}_{F^r, Z}} \left( M' , \mc{O}_Z \right)
\end{eqnarray}
is an isomorphism.
\end{theorem}

\begin{proof}
For any torsion-free $\mc{N} \in \Ob \lmod^{fu} \left( Z, \mc{O}_{F^r, Z} \right)$
let $\mc{N}^{vec}$ denote the subsheaf generated by $\Gamma ( 
Z, \mc{N} )^{vec}$ (see Section~\ref{unitstruct}).  For
any $N \in \Ob \mod^c \left( Z , \mb{F}_{p^r} \right)$, let $N^{con}
\subseteq N$ denote the subsheaf generated by the global sections of 
$N$.  The reader may check the following observations:
\begin{enumerate}
\item For any torsion-free finitely-generated unit $\mc{O}_{F^r, Z}$-module 
$\mc{N}$, the sheaf
\begin{eqnarray}
\sheafHom_{\mc{O}_{F^r, Z}} \left( \mc{N} , \mc{O}_Z \right)^{con}
\end{eqnarray}
is the sheaf of $\mc{O}_{F^r, Z}$-homomorphisms from $\mc{N}$ into $\mc{O}_Z$ that
kill $\mc{N}^{vec}$.

\item For any constructible $\mb{F}_{p^r}$-\'etale sheaf $N$ on
$Z$, the sheaf
\begin{eqnarray}
\sheafHom_{\mb{F}_{p^r}} \left( N , \mc{O}_Z \right)^{vec}
\end{eqnarray}
is the sheaf of $\mb{F}_{p^r}$-homomorphisms from $N$ into $\mc{O}_Z$
that kill $N^{con}$.
\end{enumerate}

This symmetry has a number of useful consequences.  The 
sheaf ${\mc{M}'}^{vec}$ is isomorphic to 
\begin{eqnarray}
\sheafHom_{\mc{O}_{F^r, Z}} \left( M / M^{con} , \mc{O}_Z \right).
\end{eqnarray}
The quotient $M / M^{con}$ is simply the pushforward of an
\'etale sheaf on $\{ \eta \}$.  The Riemann-Hilbert correspondence
over $\{ \eta \}$ implies that ${\mc{M}'}^{vec}$ is a finitely-generated
unit $\mc{O}_{F^r, Z}$-module.  Meanwhile, the quotient sheaf
$\mc{M}' / {\mc{M}'}^{vec}$ is isomorphic to
\begin{eqnarray}
\sheafHom_{\mb{F}_{p^r}} \left( M^{con} , \mc{O}_Z \right).
\end{eqnarray}
Since $M^{con}$ is a constant sheaf, this sheaf is simply
isomorphic to $\mc{O}_Z^{\oplus d}$ for some $d$.
Thus there is an exact sequence
\begin{eqnarray}
0 \to {\mc{M}'}^{vec} \to \mc{M}' \to \mc{O}_Z^{\oplus d} \to 0,
\end{eqnarray}
which implies (by Corollary~\ref{fgustruct}) that $\mc{M}'$
is a finitely-generated unit $\mc{O}_{F^r, Z}$-module.

Let
\begin{eqnarray}
M'' = \sheafHom_{\mc{O}_{F^r, Z}} \left(
\mc{M}' , \mc{O}_Z \right)
\end{eqnarray}
be the double-dual of $M$.
The symmetry discussed above makes the sheaf ${M''}^{con}$ naturally
isomorphic to the double-dual of the sheaf $M^{con}$, and makes
$M'' / {M''}^{con}$ naturally isomorphic to the double-dual
of the sheaf $M / M^{con}$.  There
are double-dual maps
\begin{eqnarray}
M^{con} \to {M''}^{con} \textnormal{ and } M / M^{con} \to M'' / {M''}^{con}.
\end{eqnarray}
It is easily seen that $M^{con}$ is isomorphic to
its double-dual.  The same is true for $M / M^{con}$ by the
Riemann-Hilbert correspondence over $\{ \eta \}$. 
Thus in the diagram
\begin{eqnarray}
\xymatrix{
0 \ar[r] & M^{con} \ar[r] \ar[d]  & M \ar[r] \ar[d]
& M / M^{con} \ar[r] \ar[d] & 0 \\
0 \ar[r] & {M''}^{con} \ar[r] & M'' \ar[r]
& M'' / {M''}^{con} \ar[r] & 0,
}
\end{eqnarray}
both of the outside vertical maps are isomorphisms.  The
homomorphism $M \to M''$ must be an isomorphism by the $5$-lemma.

The proof of the second part of Theorem~\ref{localRH} proceeds
similarly.  One needs only the additional fact that
$\mc{M} / \mc{M}^{vec}$ is isomorphic to $\mc{O}_Z^{\oplus e}$
for some $e > 0$.  (This is implied by Proposition~\ref{trivquot}.)
\end{proof}

\subsection{Roots of unit $A[F^r]$-modules}

\label{localrootssubsection}

We revert to algebraic notation.  Let $W$ be a finitely-generated
unit $A[F^r]$-module which has no $A$-torsion.
Then an $A$-submodule $W_0 \subseteq W$ is a root if:
\begin{enumerate}
\item $W_0$ is a finitely-generated $A$-module,
\item \label{rootprop} the $A$-submodule generated by $F^r ( W_0 ) \subseteq W$ contains $W_0$, and
\item $W$ is generated as a left $A[F^r]$-module by $W_0$.
\end{enumerate}
As in Proposition~\ref{rootfiltration}, a root determines a filtration
\begin{eqnarray}
W_0 \subseteq W_1 \subseteq W_2 \subseteq \ldots
\end{eqnarray}
for $W$, in which $W_i$ is the $A$-submodule of $W$ generated by
$F^{ri} ( W_0 )$.  Each of these modules has
an isomorphism $W_i \cong F^{ri*}_A W_0$ given by
the structural morphism of $W$.  (Here $F^{ri*}_A W_0$ denotes
the tensor product $A \otimes_A W_0$ taken via $F^{ri}_A \colon A \to A$.)
Each module $W_i$ is a free $A$-module with rank equal to $\dim_K K
\otimes_A W$.

Our goal in this subsection is to establish a useful exact sequence that
involves the dual of a root.

Suppose that $W_0$ is a root
for $W$.  Let $\{ w_1, \ldots, w_n \}$ be an $A$-module
basis for $W_0$.  By property (\ref{rootprop}) above,
each element of the basis may be expressed (uniquely, in fact)
as
\begin{eqnarray}
w_i = \sum_{j=1}^n a_{ij} F^r ( w_j),
\end{eqnarray}
with $a_{ij} \in A$.

Let
\begin{eqnarray}
W_0^\vee = \Hom_A ( W_0, A)
\end{eqnarray}
be the $A$-module dual of $W_0$.  The module $W_0^\vee$ has
a canonical left $A$-module structure: if $\phi \colon W \to A$ is
any $A$-module homomorphism, we define $F^r ( \phi )$ to be the composition
of the diagram
\begin{eqnarray}
\xymatrix{ W_0 \ar[r] & W_1 \ar[r]^\cong &
F^{r*}_A W_0 \ar[rrr]^{(a \otimes w) \mapsto a F^r ( \phi ( w ) )}
& & & A. }
\end{eqnarray}
The left $A[F^r]$-module structure of $W_0^\vee$ may also be understood in terms
of the basis $\{ w_1, \ldots, w_n \}$ chosen above.  Let $\{ w_1^\vee ,
\ldots , w_n^\vee \}$ be the dual of this basis.  Then (as the reader may check),
\begin{eqnarray}
F^r ( w_i^\vee ) = \sum_{j=1}^n a_{ji} w_j^\vee.
\end{eqnarray}

Suppose that $\psi$ is an element of $W_0^\vee$ which is invariant
under the action of $F^r$.  Then, the composite map
\begin{eqnarray}
\xymatrix{ W_1 \ar[r]^\cong &
F^{r*}_A W_0 \ar[rrr]^{(a \otimes w) \mapsto a F^r ( \psi ( w ) )}
& & & A }
\end{eqnarray}
is compatible with the map $\psi \colon W_0 \to A$ itself.  In fact,
there is a sequence of induced maps
\begin{eqnarray}
\xymatrix{ W_i \ar[r]^\cong &
F^{ri*}_A W_0 \ar[r]
& A. }
\end{eqnarray}
(for $i = 1, 2, \ldots$) all of which are compatible via restriction.  
Taken together, these maps determine a left $A[F^r]$-module homomorphism
from $W$ into $A$.  In this way we see that the $F^r$-invariant elements
of $W_0^\vee$ are precisely the elements that arise by restriction
of maps from $\Hom_{A[F^r]} \left( W, A \right)$.

\begin{proposition}
\label{invariantsandroots}
Let $W$ be a finitely-generated unit $A[F^r]$-module that has no $A$-torsion.
Suppose that $W_0 \subseteq W$ is a root of $W$.  Then restriction determines
a bijection
\begin{eqnarray}
\Hom_{A[F^r]} \left( W, A \right) \to \left( W_0^\vee \right)^{F^r}.
\end{eqnarray}
\end{proposition}

This proposition has the following consequence.

\begin{theorem}
\label{localRHalt}
Let $W$ be an object from $\lmod^{fu} \left( A[F^r] \right)$ which is
a torsion-free $A$-module.  Suppose that $W$ has a root $W_0$.  Let
$M = \Hom_{A[F^r]} \left( W, A \right)$.  Then the sequence
\begin{eqnarray}
\label{exactseqroot}
\xymatrix{
0 \ar[r] & M \ar[r] & W_0^\vee \ar[r]^{(1 - F^r)} &
W_0^\vee \ar[r] & 0
}
\end{eqnarray}
is exact.
\end{theorem}

\begin{proof}
We need only to show that the action of $1 - F^r$ on $W_0^\vee$ is
surjective.  This is accomplished by the following lemma.

\begin{lemma}
\label{surjlemma}
For any $v \in W_0^\vee$, there exists an element $v' \in
W_0^\vee$ such that $v' - F^r ( v' ) = v$.
\end{lemma}

\begin{proof}
Using the notation from earlier in this subsection, we may write
\begin{eqnarray}
v = \sum_{i=1}^n b_i w^\vee_i
\end{eqnarray}
with $b_i \in A$.
Solving the equation $v' - F^r ( v' ) = v$ amounts to finding elements
$b'_1 , \ldots, b'_n \in A$ such that
\begin{eqnarray}
\sum_{i=1}^n b'_i w_i^\vee - F^r \left( \sum_{i=1}^n 
b'_i w_i^\vee \right) = \sum_{i=1}^n b_i w_i^\vee,
\end{eqnarray}
which is equivalent to solving the system of equations
\begin{eqnarray}
b'_i - \sum_{j=1}^n (b_j')^{p^r} a_{ji} =
b_i
\end{eqnarray}
($i = 1, 2, \ldots, n$).  This in turn is equivalent
to finding a homomorphism of the $A$-algebra
\begin{eqnarray}
S = A[Y_1, \ldots, Y_n] / \left(
\left\{ Y_i - \sum_{j=1}^n Y_j^{p^r} a_{ji} 
- b_i \right\}_{i=1}^n
\right)
\end{eqnarray}
into $A$.  By calculating the module of relative differentials
of $S$ over $A$ one sees that $S$ is a finite \'etale $A$-algebra,
and is therefore simply isomorphic to a direct sum of copies of $A$. 
Thus sections $S \to A$ clearly exist, and the lemma is proved.
\end{proof}

Proposition~\ref{invariantsandroots} and Lemma~\ref{surjlemma}
together prove Theorem~\ref{localRHalt}.
\end{proof}

The case of Theorem~\ref{localRHalt} that is of interest
is when $W$ and $M$ are stalks of sheaves related by the
Riemann-Hilbert correspondence.  This case will come up in 
Section~\ref{RHcurvesection}.

\subsection{The minimal root index}

The concept of a root leads to the definition of an invariant
which measures the complexity of objects from $\lmod^{fu} \left(
A [F^r] \right)$.  The basis for the invariant is the following theorem, which
is a special case of a result of M.~Blickle:

\begin{theorem}
Let $W$ be finitely-generated unit $A[F^r]$-module.  Then $W$
has a unique root $W_0$ which is contained in every other root.
\end{theorem}

\begin{proof}
See Theorem 2.10 from \cite{blickleintersection}.
\end{proof}

\begin{definition}
\label{minimalrootindexdef}
Let $W$ be an object from $\lmod^{fu} \left( A [F^r] \right)$
which is a torsion-free $A$-module.  Let $W_0$ be
the minimal root in $W$, and let $W_1 \subseteq W$ be
the $A$-submodule generated by $F^r ( W_0 )$.  The
\textnormal{minimal root index} of $W$ is
\begin{eqnarray}
\frac{\dim_k \left( W_1 / W_0 \right)}{p^r - 1}.
\end{eqnarray}
Let $M$ be a constructible $\mb{F}_{p^r}$-\'etale sheaf on $\Spec A$.
Then the minimal root index of $M$ is the minimal root
index of
\begin{eqnarray}
\Hom_{\mb{F}_{p^r}} \left( M , \mc{O}_{\Spec A} \right)
\end{eqnarray}
(the dual of $M$ under the Riemann-Hilbert correspondence).
\end{definition}

Note that the minimal root index is always finite.  (Since
$W_1$ and $W_0$ have the same rank as $A$-modules, $W_1 / W_0$
is a torsion $A$-module, and is therefore finite-dimensional
over $k$.)  However it is not necessarily
integral, as the following example calculation shows.

\begin{example}
\label{quadraticexample}
Let $W$ be a $K$-vector space generated by a single element $e$,
and define the left $A[F^r]$-module structure on $W$ by
\begin{eqnarray}
F^r ( e ) = t^{(p^r - 1)/2} e.
\end{eqnarray}
Finitely-generated $A$-submodules of $W$ are all of the
form $A(t^N e)$, with $N \in \mathbb{Z}$.  The smallest such module
which satisfies the root properties is $A(t^{-1} e)$.  In this
case
\begin{eqnarray}
W_0 = A ( t^{-1} e ) \textit{ and } W_1 = A ( t^{(-p^r - 1)/2} e).
\end{eqnarray}
The $k$-dimension of $W_1 / W_0$ is $(p^r - 1)/2$, and the minimal
root index of $W$ is $\frac{1}{2}$.

Note that in this case, the Riemann-Hilbert dual of $W$ is an
$\mathbb{F}_{p^r}$-\'etale sheaf on $\Spec A$ of generic rank $1$.
The dual of $W$ is a functor which associates
to any \'etale ring extension $A \to B$ the vector space 
$\Hom_{B[F^r]} \left( B \otimes W , B \right)$.  This sheaf has
nontrivial sections, for example, for the $k$-algebra homomorphism
$A \to K$ which maps $t$ to $t^2$.
\end{example}

The next proposition follows easily from Proposition~\ref{Amodtriv}.
The proof is left to the reader.

\begin{proposition}
\label{indexzero}
Let $W$ be an object from $\lmod^{fu} \left( A[F^r] \right)$
which is a free finite-rank $A$-module.  Then the minimal root
for $W$ is $W$ itself.  The minimal root index for $W$ is zero.
\end{proposition}

\section{The Riemann-Hilbert correspondence on a curve}

\label{RHcurvesection}

Throughout this section, let $X$ be a smooth $k$-curve.
The Riemann-Hilbert
correspondence over $X$ relates $\mb{F}_{p^r}$-\'etale sheaves on $X$
to unit $\mc{O}_{F^r, X}$-modules.  Some relationships between
\'etale cohomology and coherent cohomology can be deduced
from the correspondence.  In this section we will develop
the Riemann-Hilbert correspondence over $X$ by building
on the local results from Section~\ref{localsection}.

For any smooth $k$-scheme $Z$, let $\mod^c \left( Z , \mb{F}_{p^r}
\right)$ denote the category of constructible
$\mb{F}_{p^r}$-\'etale sheaves on $Z$.  If $z$ is a closed point
of $Z$, let $\mc{O}_{Z, z}$ denote the stalk of the \'etale
coordinate sheaf of $Z$.  If $Q$ is an \'etale sheaf on $Z$, let
$Q_{(z)}$ denote the pullback of $Q$ via the natural morphism
\begin{eqnarray}
\Spec \mc{O}_{Z, z} \to Z.
\end{eqnarray}

\subsection{The functor $\sheafHom_{\mb{F}_{p^r}} \left( \cdot , 
\mc{O}_X \right)$}

If $M$ is an $\mb{F}_{p^r}$-\'etale sheaf on $X$, then the sheaf
$\sheafHom_{\mb{F}_{p^r}} \left( M , \mc{O}_X \right)$
has a left $\mc{O}_{F^r, X}$-module structure given by the 
left $\mc{O}_{F^r, X}$-module structure of $\mc{O}_X$.
The next proposition shows that the functor $\sheafHom_{\mb{F}_{p^r}}
\left( \cdot , \mc{O}_X \right)$ is compatible with the
analogous functor from the local Riemann-Hilbert correspondence
(see Theorem~\ref{localRH}).

\begin{proposition}
\label{stalkiso1}
Let $M$ be an object of $\mod^c ( X, \mb{F}_{p^r} )$. 
Let $x$ be a closed point of $X$.
The natural homomorphism
\begin{eqnarray}
\sheafHom_{\mb{F}_{p^r}} \left( M , \mc{O}_X \right)_x
\to \Hom_{\mb{F}_{p^r}} \left( M_{(x)} , 
\mc{O}_{\Spec \mc{O}_{X, x}} \right)
\end{eqnarray}
is an isomorphism.
\end{proposition}

We will prove Proposition~\ref{stalkiso1} by reduction
to the following special case.

\begin{proposition}
\label{stalkiso1specialcase}
Let $Y$ be a smooth affine curve over $k$.  Let
$Z \to Y$ be a finite Galois cover which is 
totally ramified at one closed
point $y \in \left| Y \right|$ and unramified elsewhere.
Let $\{ z \}$ be the pre-image of $\{ y \}$ in $Z$, and
let $Z' \subseteq Z$ be the complement of $\{ z \}$.
Let $N$ be a constructible $\mb{F}_{p^r}$-\'etale sheaf
on $Y$ such that $N_{\mid Z'}$ is constant.  Then
the natural homomorphism
\begin{eqnarray}
\label{specialstalkisoformula}
\sheafHom_{\mb{F}_{p^r}} \left( N , \mc{O}_Y \right)_y
\to \Hom_{\mb{F}_{p^r}} \left( N_{(y)} , \mc{O}_{\Spec 
\mc{O}_{Y, y}} \right)
\end{eqnarray}
is an isomorphism.
\end{proposition}

{\it Proof.} The curves $Y$, $Z$, and $Z'$ are all affine.
Let $Y = \Spec R$, $Z = \Spec S$, and 
$Z' = \Spec S'$.  If $\Spec Q \to \Spec R$ is any \'etale
morphism, then morphisms
\begin{eqnarray}
N_{ \mid \Spec Q } \to \mc{O}_{\Spec Q}
\end{eqnarray}
may be expressed as
commutative diagrams
\begin{eqnarray}
\xymatrix{
N ( \Spec S' ) \ar[r]^{\phi} & S' \otimes_R Q \\
N ( \Spec R ) \ar[u]^{\rho_1} \ar[r]^{\psi} & Q \ar[u]^{\rho_2} \\
}
\end{eqnarray}
in which $\rho_1$ and $\rho_2$ are sheaf restriction maps,
$\phi$ and $\psi$ are $\mb{F}_{p^r}$-linear homomorphisms, and 
$\phi$ is $\Aut ( S' / R )$-equivariant.
Similarly, morphisms
\begin{eqnarray}
N_{(y)} \to \mc{O}_{\Spec \mc{O}_{Y, y}}
\end{eqnarray}
may be expressed as commutative diagrams
\begin{eqnarray}
\xymatrix{
N( \Spec S' ) \ar[r] & S' \otimes_R \mc{O}_{Y, y} \\
N( \Spec R ) \ar[r] \ar[u] & \mc{O}_{Y, y } \ar[u] \\
}
\end{eqnarray}
in which the vertical maps are sheaf restriction maps, the
horizontal maps are $\mb{F}_{p^r}$-linear homomorphisms, and
the top map is $\Aut ( S' / R )$-equivariant.

Suppose that 
\begin{eqnarray}
\label{morphismdiagram}
\xymatrix{
N( \Spec S' ) \ar[r] & S' \otimes_R \mc{O}_{Y, y} \\
N( \Spec R ) \ar[r] \ar[u] & \mc{O}_{Y, y } \ar[u] \\
}
\end{eqnarray}
is the diagram for a morphism $N_{(y)} \to 
\mc{O}_{\Spec \mc{O}_{Y, y}}$.  
Since $N ( \Spec R )$ and $N ( \Spec S' )$ are
finite, there exists an \'etale $R$-algebra
$P \subseteq \mc{O}_{Y, y}$ such that 
the images of $N( \Spec R )$ and $N ( \Spec S' )$
are contained in $P$ and $S' \otimes_R P$, respectively.
Thus (\ref{morphismdiagram}) determines a commutative
diagram
\begin{eqnarray}
\xymatrix{
N ( \Spec S' ) \ar[r] & S' \otimes_R P \\
N ( \Spec R ) \ar[u] \ar[r] & P. \ar[u] \\
}
\end{eqnarray}
We conclude that any $\mb{F}_{p^r}$-linear morphism
$N_{(y)} \to \mc{O}_{\Spec \mc{O}_{Y, y}}$ extends
to an $\mb{F}_{p^r}$-linear morphism from 
$N \to \mc{O}_Y$ on some \'etale neighborhood
of $y$.  We have constructed an inverse to homomorphism 
(\ref{specialstalkisoformula}). $\Box$

{\it Proof of Proposition~\ref{stalkiso1}.}
Replacing $X$ with an open subcurve if necessary,
we may assume that $X$ is affine and that the
constructible sheaf 
$M$ is locally constant on $X \smallsetminus
\{ x \}$.  Let $Z' \to X \smallsetminus \{ x \}$
be a finite Galois \'etale cover such that
$M_{\mid Z'}$ is constant.

Let $K_0 / K(X)$ be the largest subextension
of $K(Z') / K(X)$ that is unramified at $x$.  The
tower of field extensions
\begin{eqnarray}
K(Z') \supseteq K_0 \supseteq K(X)
\end{eqnarray}
determines a diagram of smooth projective curves,
\begin{eqnarray}
\label{diagramofprojcurves}
\xymatrix{\overline{Z'} \ar[r] \ar[rd] &
W \ar[d] \\
& \overline{X} \\
}
\end{eqnarray}
where $\overline{Z'}$ and $\overline{X}$ are the
smooth projective closures of $Z'$ and $X$, respectively,
and $W$ is the unique smooth projective
curve over $k$ whose fraction field is $K_0$.
Let
\begin{eqnarray}
\xymatrix{\overline{Z'} \times_{\overline{X}} X
\ar[r] \ar[rd] &
W \times_{\overline{X}} X \ar[d] \\
& X \\
}
\end{eqnarray}
be the diagram obtained from (\ref{diagramofprojcurves})
via base change.
The morphism
\begin{eqnarray}
\overline{Z'} \times_{\overline{X}} X \to
X
\end{eqnarray}
is Galois and finite, and \'etale away from $x$.
The morphism
\begin{eqnarray}
W \times_{\overline{X}} X \to X
\end{eqnarray}
is Galois and finite and \'etale at all points of $X$.
The morphism
\begin{eqnarray}
\overline{Z'} \times_{\overline{X}} X 
\to W \times_{\overline{X}} X
\end{eqnarray}
is Galois and finite, \'etale away from $x$, 
and totally ramified at $x$.
Proposition~\ref{stalkiso1specialcase} may therefore be
applied with
$Y = W \times_{\overline{X}} X$, $Z = \overline{Z'} \times_{\overline{X}}
X$, and $N = M_{\mid W \times_{\overline{X}} X}$.
This application completes the current proof, since the assertion of
Proposition~\ref{stalkiso1} is local at $x$.  $\Box$

\begin{proposition}
\label{isfiniteunit}
Let $M$ be a constructible $\mb{F}_{p^r}$-\'etale sheaf
on $X$ whose sections all have open support.  Let 
\begin{eqnarray}
\mc{M} = \sheafHom_{\mb{F}_{p^r}} \left( M , \mc{O}_X \right),
\end{eqnarray}
with a left $\mc{O}_{F^r, X}$-module structure given by
the Frobenius endomorphism of $\mc{O}_X$.  Then $\mc{M}$
is an lfgu (locally finitely-generated unit)
$\mc{O}_{F^r, X}$-module, and it is torsion-free
as an $\mc{O}_X$-module.
\end{proposition}

{\it Proof.} The proof of this proposition consists of three
lemmas.

\begin{lemma}
The sheaf $\mc{M}$ is a quasi-coherent $\mc{O}_X$-module.
\end{lemma}

{\it Proof.} Let $U \subseteq X$ be a nonempty open
subset on which $M$ is locally constant.  Let $j
\colon U \to X$ be the inclusion morphism.  The sheaf
\begin{eqnarray}
\sheafHom_{\mb{F}_{p^r}} \left( M_{\mid U} ,
\mc{O}_U \right)
\end{eqnarray}
is locally free of finite rank as an $\mc{O}_U$-module.
The pushforward
\begin{eqnarray}
j_* 
\sheafHom_{\mb{F}_{p^r}} \left( M_{\mid U} ,
\mc{O}_U \right) \cong 
\sheafHom_{\mb{F}_{p^r}} \left( M ,
j_* \mc{O}_U \right)
\end{eqnarray}
is a quasi-coherent $\mc{O}_X$-module.  There is
a natural morphism
\begin{eqnarray}
\label{morphismfromM}
\mc{M} \hookrightarrow \sheafHom_{\mb{F}_{p^r}} \left(
M , j_* \mc{O}_U \right).
\end{eqnarray}
To show that $\mc{M}$ is quasi-coherent, it suffices
to show that the cokernel of this morphism is quasi-coherent.

The image of (\ref{morphismfromM}) consists of the morphisms
$M \to j_* \mc{O}_U$ that map $M_x$ into $\mc{O}_{X, x}$
for each $x \in \left| X \smallsetminus U \right|$.  Suppose
that $x$ is an element of $\left| X \smallsetminus U \right|$,
and suppose that $\phi$ is a morphism from $M$ to
$j_* \mc{O}_U$ over a Zariski open neighborhood of $x$.
Choose a local parameter $t$ at $x$.  Since $M_x$
is finite, we may choose $n$ sufficiently large
so that $t^n \phi$ maps $M_x$ into $\mc{O}_{X, x}$.
We conclude that the cokernel of (\ref{morphismfromM}) is a
quasi-coherent skyscraper sheaf supported at $\left| X \smallsetminus 
U \right|$.  This completes the proof.  $\Box$

\begin{lemma}
The structural morphism
\begin{eqnarray}
F_X^* \mc{M} \to \mc{M}
\end{eqnarray}
is an isomorphism.
\end{lemma}

{\it Proof.} It suffices to show that for any closed point
$x \in \left| X \right|$, the structural morphism of 
$\mc{M}_x$ is an isomorphism.  By Proposition~\ref{stalkiso1},
there exist isomorphisms
\begin{eqnarray}
\mc{M}_x
\to \Hom_{\mb{F}_{p^r}} \left( M_{(x)} , 
\mc{O}_{\Spec \mc{O}_{X, x}} \right)
\end{eqnarray}
for each closed point $x \in \left| X \right|$.
By Theorem~\ref{localRH}, each left $\mc{O}_{X, x}[F^r]$-module
\begin{eqnarray}
\Hom_{\mb{F}_{p^r}} \left( M_{(x)} , \mc{O}_{\Spec \mc{O}_{X, x}} \right)
\end{eqnarray}
is a unit $\mc{O}_{X, x}[F^r]$-module.
$\Box$

\begin{lemma}
The left $\mc{O}_{F^r, X}$-module $\mc{M}$ is generated
by a finite number of sections.
\end{lemma}

{\it Proof.} Let $U \subseteq X$ be a Zariski open subset
on which $M$ is locally constant.  By Proposition~\ref{stalkiso1}
and Theorem~\ref{localRH}, each stalk $\mc{M}_x$ is a finitely-generated unit
$\mc{O}_{X, x}[F^r]$-module.  For each point $x \in \left| 
X \smallsetminus U \right|$, choose a finite set
of sections of $\mc{M}$ on a Zariski
open neighborhood of $x$ which generate $\mc{M}_x$ as a
left $\mc{O}_{X, x}[F^r]$-module.  Choose a finite set
of sections of the coherent $\mc{O}_U$-module $\mc{M}_{\mid U}$
which generate $\mc{M}_{\mid U}$ as an $\mc{O}_U$-module.
Let $\mc{M}' \subseteq \mc{M}$ be the sub-left-$\mc{O}_{F^r, X}$-module
generated by all of the aforementioned sections.  The stalk
of $\mc{M}'$ at any closed point $x$ is equal to 
$\mc{M}$.  Therefore $\mc{M}' = \mc{M}$.  $\Box$

It is clear that $\mc{M}$ is torsion-free as an $\mc{O}_X$-module.
The proof of Proposition~\ref{isfiniteunit} is complete.
$\Box$

\subsection{The functor $\sheafHom_{\mc{O}_{F^r, X}} \left( \cdot ,
\mc{O}_X \right)$}

If $\mc{M}$ is a unit $\mc{O}_{F^r, X}$-module, then the sheaf
of homomorphisms
\begin{eqnarray}
\sheafHom_{\mc{O}_{F^r, X}} \left( \mc{M} , \mc{O}_X \right)
\end{eqnarray}
is a sheaf of $\mb{F}_{p^r}$-vector spaces on $X$.

\begin{proposition}
\label{stalkiso2}
Let $\mc{M}$ be an object of $\lmod^{fu} ( 
X , \mc{O}_{F^r, X} )$.  Let $x$ be a closed point
of $X$.
The natural homomorphism
\begin{eqnarray}
\label{stalkiso2exp}
\sheafHom_{\mc{O}_{F^r, X}} ( \mc{M} , \mc{O}_X )_x
\to \Hom_{\mc{O}_{X, x}[F^r]} ( \mc{M}_x ,
\mc{O}_{X, x} )
\end{eqnarray}
is an isomorphism.
\end{proposition}

{\it Proof.} It is clear that (\ref{stalkiso2exp})
is injective.  To prove the proposition it suffices
to show that any element of
\begin{eqnarray}
\Hom_{\mc{O}_{X, x}[F^r]} ( \mc{M}_x ,
\mc{O}_{X, x} )
\end{eqnarray}
may be extended to a left $\mc{O}_{F^r, X}$-module
homomorphism from $\mc{M}$ to $\mc{O}_X$ over an \'etale
neighborhood of $x$.

Let
\begin{eqnarray}
\phi \colon \mc{M}_x \to \mc{O}_{X, x}
\end{eqnarray}
be a left $\mc{O}_{X, x}[F^r]$-module homomorphism.
Let $U \subseteq X$ be an affine
neighborhood of $x$.  Let $R = \Gamma ( U ,
\mc{O}_U )$ and $P = \Gamma ( U , \mc{M} )$.
Let $\left\{ p_1, \ldots, p_c \right\} \subseteq
P$ be a subset which generates $P$ as a left
$R[F^r]$-module.  Choose an \'etale $R$-algebra
$R' \subseteq \mc{O}_{X, x}$ large enough
to contain the images of the stalks of each
$p_i$ under $\phi$.  The homomorphism
\begin{eqnarray}
\mc{O}_{X, x} \otimes_R P \to \mc{O}_{X, x}
\end{eqnarray}
determined by $\phi$ restricts to a homomorphism
\begin{eqnarray}
R' \otimes_R P \to R'.
\end{eqnarray}
Thus there is a homomorphism from $\mc{M}_{\mid \Spec R'}$ 
to $\mc{O}_{\Spec R'}$ whose stalk is $\phi$.  $\Box$

\begin{proposition}
\label{etaleconstant}
Let $\mc{M}$ be an object of $\lmod^{fu} ( X , 
\mc{O}_{F^r, X} )$.  Let
\begin{eqnarray}
M = \sheafHom_{\mc{O}_{F^r, X}} ( \mc{M} ,
\mc{O}_X ).
\end{eqnarray}
Then there exists a nonempty \'etale $X$-scheme
$V$ such that $M_{\mid V}$
is a constant $\mb{F}_{p^r}$-sheaf of finite rank.
\end{proposition}

{\it Proof.}
By Proposition~\ref{densecoherent}, there exists a nonempty
open subset $X' \subseteq X$ such that $\mc{M}_{\mid X'}$
is a coherent $\mc{O}_{X'}$-module.
Let $\alpha \in \left| X' \right|$
denote the generic point, and let
\begin{eqnarray}
\overline{\alpha} \colon \Spec \overline{k ( \alpha 
) } \to X'
\end{eqnarray}
denote a geometric point at $\alpha$.
The geometric
stalk $\mc{M}_{\overline{\alpha}}$ has a 
$\overline{k( \alpha )}$-basis that is fixed by
$F^r$ (by Proposition~\ref{sepclosedfrobenius}).
Choose an \'etale scheme $U$ over $X'$ on
which there exist representatives
\begin{eqnarray}
m_1 , \ldots , m_e \in
\mc{M}(U)
\end{eqnarray}
for the elements of this basis.
The coherent subsheaf of $\mc{M}_{\mid U}$
generated by $\left\{ m_i \right\}_{i=1}^e$
has the same generic rank as $\mc{M}_{\mid U}$.  Let
$V \subseteq U$ be a nonempty open subset on which these
two sheaves are equal.  Then $\mc{M}_{\mid V}$ is
isomorphic as a left $\mc{O}_{F^r, X}$-module to
$\mc{O}_V^{\oplus e}$.  Therefore $M_{\mid V}$ is
isomorphic to a constant $\mb{F}_{p^r}$-\'etale sheaf
of rank $e$.  $\Box$

\begin{corollary}
Let $\mc{M}$ be an object of $\lmod^{fu} ( X , 
\mc{O}_{F^r, X} )$ which is torsion-free as
an $\mc{O}_X$-module.  Then
\begin{eqnarray}
M = \sheafHom_{\mc{O}_{F^r, X}} ( \mc{M} ,
\mc{O}_X )
\end{eqnarray}
is a constructible $\mb{F}_{p^r}$-\'etale sheaf
on $X$.  
\end{corollary}

{\it Proof.} Proposition~\ref{etaleconstant}
implies that $M$ is locally constant on a nonempty
open subset of $X$.  Theorem~\ref{localRH} implies
(via Proposition~\ref{stalkiso2}) that the stalks
of $M$ are finite.~$\Box$

\begin{proposition}
\label{globaldoubledual1}
Let $\mc{M}$ be a torsion-free lfgu $\mc{O}_{F^r, X}$-module.
The double-dual
homomorphism
\begin{eqnarray}
\label{ddmorphism1}
\mc{M} \to \sheafHom_{\mb{F}_{p^r}} \left(
\sheafHom_{\mc{O}_{F^r , X}} \left( \mc{M} , \mc{O}_X 
\right) , \mc{O}_X \right)
\end{eqnarray}
is an isomorphism.
\end{proposition}

\begin{proposition}
\label{globaldoubledual2}
Let $M$ be a constructible $\mb{F}_{p^r}$-\'etale sheaf
on $X$ whose sections all have open support.  The
double-dual homomorphism
\begin{eqnarray}
\label{ddmorphism2}
M \to \sheafHom_{\mc{O}_{F^r, X}} \left(
\sheafHom_{\mb{F}_{p^r}} \left( M , \mc{O}_X \right) ,
\mc{O}_X \right)
\end{eqnarray}
is an isomorphism.
\end{proposition}

{\it Proof of Propositions~\ref{globaldoubledual1}
and \ref{globaldoubledual2}.} It suffices to show that morphisms
(\ref{ddmorphism1}) and (\ref{ddmorphism2}) 
induce isomorphisms on closed stalks.  This
assertion follows from Theorem~\ref{localRH} via Propositions~\ref{stalkiso1}
and \ref{stalkiso2}. $\Box$

\subsection{Roots on curves}

\label{globalrootssubsection}

The next two theorems globalize results on roots from Section~\ref{localsection}.

\begin{theorem}
\label{minrootthm}
Let $\mc{M}$ be a torsion-free lfgu $\mc{O}_{F^r, X}$-module.  Then
$\mc{M}$ has a unique minimal root $\mc{M}_0$ which is contained in
every other root.  For any closed point
$x \in \left| X \right|$, the stalk of $\mc{M}_0$ at $x$ is the
minimal root of $\mc{M}_x$.
\end{theorem}

\begin{proof}
Our method is to define a subsheaf $\mc{M}_0$ and then prove that
it has the desired properties.
For any \'etale morphism $V \to X$, let $\mc{M}_0 (V) \subseteq \mc{M}(V)$
be the subset consisting of sections $m \in \mc{M} ( V)$ such that
for any closed point $x \in \left| X \right|$ and any diagram
\begin{eqnarray}
\xymatrix{ \Spec k(x) \ar[r] \ar[rd] & V \ar[d] \\
& X,}
\end{eqnarray}
the stalk element at $x$ represented by $m$ is contained in the minimal root
of $\mc{M}_x$.

\begin{lemma}
For any closed point $x \in \left| X \right|$, $\left( \mc{M}_0 \right)_x$
is equal to the minimal root of $\mc{M}_x$.
\end{lemma}

\begin{proof}
By Proposition~\ref{densecoherent}, there is a nonempty open subcurve
$U \subseteq X$ on which $\mc{M}$ is coherent.  Since $\mc{M}$ is
also torsion-free, this makes $\mc{M}_{\mid U}$ a locally free
$\mc{O}_U$-module of finite rank.  By Proposition~\ref{indexzero},
the minimal root of $\mc{M}_y$ at any closed point $y \in \left| U \right|$
is $\mc{M}_y$ itself.  Thus the condition which defines $\mc{M}_0$
above needs only to be checked at points outside of $U$.

Let $x$ be a closed point of $\left| X \right|$.  We show that the stalk
$\left( \mc{M}_0 \right)_x$ contains
the minimal root of $\mc{M}_x$.  Choose any
element $m_x$ from the minimal root of $\mc{M}_x$.  There exists
an \'etale neighborhood
\begin{eqnarray}
\xymatrix{ \Spec k \ar[r] \ar[rd] & V \ar[d] \\
& X,}
\end{eqnarray}
and a section $m \in \mc{M}(V)$ which represents $m_x$.  Let $x'
\in \left| V \right|$ be the image of $\Spec k$ in the diagram above.
Consider the restriction
\begin{eqnarray}
m_{\mid (V \times_X U) \cup \{ x' \}  } \in
\mc{M} ( (V \times_X U ) \cup \{ x' \} ).
\end{eqnarray}
By definition, this section is contained in the subsheaf $\mc{M}_0$.  Therefore
its stalk
$m_x$ is contained in $\left( \mc{M}_0 \right)_x$.

We have shown that the minimal root of $\mc{M}_x$ is contained
in $\left( \mc{M}_0 \right)_x$.  The reverse inclusion is obvious.
\end{proof}

\begin{lemma}
The sheaf $\mc{M}_0$ is coherent.
\end{lemma}

\begin{proof}
As in the previous proof, we may find an open subcurve
$U \subseteq X$ on which $\mc{M}$ is coherent.  For any 
$y \in \left| U \right|$, the minimal root of $\mc{M}_y$ is
$\mc{M}_y$ itself.  The quotient sheaf
$\mc{M} / \mc{M}_0$ is a quasi-coherent skyscraper sheaf
supported outside of $U$.  Since $\mc{M}$ and $\mc{M} / \mc{M}_0$
are both quasi-coherent, $\mc{M}_0$ is quasi-coherent.

Since $\mc{M}_{0 \mid U}$ is a finitely-generated $\mc{O}_U$-module,
and each stalk $\left( \mc{M}_0 \right)_x$ with $x \in \left| X \smallsetminus U \right|$
is a finitely-generated $\mc{O}_{X,x}$-module, $\mc{M}_0$ is
a finitely-generated $\mc{O}_X$-module.  Thus $\mc{M}_0$ is
coherent.
\end{proof}

The other two properties that define a root (see Definition~\ref{rootdef})
follow for $\mc{M}_0$ from the
corresponding properties for the stalks $\left( \mc{M}_0 \right)_x$.
Likewise, the fact that $\mc{M}_0$ is contained in every 
root of $\mc{M}$ follows
easily from the same property for the stalks $\left( \mc{M}_0 \right)_x$.
This completes the proof of the theorem.
\end{proof}

Suppose that $\mc{M}$ is a torsion-free lfgu $\mc{O}_{F^r, X}$-module and
$\mc{M}_0 \subseteq \mc{M}$ is a root for $\mc{M}$.  Then, as in
subsection~\ref{localrootssubsection}, we can define a left
$\mc{O}_{F^r, X}$-module structure on the coherent sheaf dual
$\mc{M}_0^\vee$.  Let $\mc{M}_0 \subseteq \mc{M}_1 \subseteq
\ldots \subseteq \mc{M}$ be the filtration of Proposition~\ref{rootfiltration}.
If $\phi \colon \mc{M}_0 \to \mc{O}_X$ is an $\mc{O}_X$-module
homomorphism, then $F^r ( \phi )$ is the composition
\begin{eqnarray}
\xymatrix{\mc{M}_0 \ar[r] & \mc{M}_1 \ar[d]^\cong & \mc{O}_X \\
& F^{r*}_X \mc{M}_0 \ar[r]^{F^{r*}_X ( \phi )} & F^{r*}_X \mc{O}_X. \ar[u]^\cong }
\end{eqnarray}

Note that if $\phi$ is nonzero, then $F^r ( \phi )$ is also nonzero.  (This
is evident because the map $\mc{M}_0 \to \mc{M}_1$ in the above diagram
is injective.)  So the action of $F^r$ on $\mc{M}_0^\vee$ is injective.

\begin{theorem}
\label{globalexactsequence}
Let $\mc{M}$ be a torsion-free lfgu $\mc{O}_{F^r, X}$-module, and let $\mc{M}_0
\subseteq \mc{M}$ be a root for $\mc{M}$.  Let
\begin{eqnarray}
M = \sheafHom_{\mc{O}_{F^r, X}} \left( \mc{M} , \mc{O}_X \right).
\end{eqnarray}
Then the map $M \to \mc{M}_0^\vee$ given by restriction fits into an
exact sequence
\begin{eqnarray}
\xymatrix{0 \ar[r] & M \ar[r] &
\mc{M}_0^\vee \ar[r]^{1 - F^r} & \mc{M}_0^\vee \ar[r] & 0.}
\end{eqnarray}
\end{theorem}

\begin{proof}
It suffices to show that the sequence
\begin{eqnarray}
\label{stalkexact}
\xymatrix{0 \ar[r] & M_x \ar[r] &
\left( \mc{M}_0^\vee \right)_x \ar[r]^{1 - F^r} & 
\left( \mc{M}_0^\vee \right)_x \ar[r] & 0.}
\end{eqnarray}
is exact for every closed point $x \in \left| X \right|$.  Note 
that $\left( \mc{M}_0^\vee \right)_x$ is canonically isomorphic
to $\left( \left( \mc{M}_0 \right)_x \right)^\vee$, and, by
Proposition~\ref{stalkiso2}, $M_x$ is canonically isomorphic to
\begin{eqnarray}
\sheafHom_{\mc{O}_{X,x}[F^r]} \left( \mc{M}_x , \mc{O}_{X,x} \right).
\end{eqnarray}
The exactness of (\ref{stalkexact}) follows from Theorem~\ref{localRHalt}.
\end{proof}

\subsection{Cohomology and the Riemann-Hilbert correspondence}

\label{cohomologysubsection}

Let $Y$ be a smooth {\it projective} $k$-curve.  Let $\mc{N}$ be a torsion-free
lfgu $\mc{O}_{F^r, Y}$-module, and let $\mc{N}_0 \subseteq \mc{N}$
be a root for $\mc{N}$.  
The (injective) Frobenius-linear endomorphism of $\mc{N}_0^\vee$ discussed
in the previous subsection
induces Frobenius-linear maps
\begin{eqnarray}
\label{cohomologymap}
H^i \left( Y, \mc{N}_0^\vee \right) \to
H^i \left( Y, \mc{N}_0^\vee \right)
\end{eqnarray}
for $i = 0, 1$.  Each cohomology group $H^i \left( Y , \mc{N}_0^\vee \right)$
is a finite-dimensional $k$-vector space.  The map (\ref{cohomologymap})
is injective for $i = 0$, and thus makes $H^i \left( Y ,
\mc{N}_0^\vee \right)$ a unit $k[F^r]$-module.  The
left $k[F^r]$-module $H^1 \left( Y , \mc{N}_0^\vee \right)$
is not necessarily unit, but it has a natural decomposition
\begin{eqnarray}
H^1 \left( Y, \mc{N}_0^\vee \right) \cong V^{ss} \oplus V^{nil},
\end{eqnarray}
where $V^{ss}$ is a unit $k[F^r]$-module and $V^{nil}$ is a left
$k[F^r]$-module with a nilpotent $F^r$-action.  (See Section~1 of
\cite{katzsga} for a discussion of this type of decomposition.)

By Proposition~\ref{sepclosedfrobenius}, the $F^r$-invariant
elements of $H^0 \left( Y , \mc{N}_0^\vee \right)$ form an
$\mb{F}_{p^r}$-subspace whose dimension is the same as the $k$-dimension
of $H^0 \left( Y, \mc{N}_0^\vee \right)$.  Likewise,
the $F^r$-invariant elements of $H^1 \left( Y, \mc{N}_0^\vee \right)$ form
an $\mb{F}_{p^r}$-vector space whose dimension is 
$\dim_k V_{ss}$.  

The Riemann-Hilbert correspondence implies that these elements are
in one-to-one correspondence with co-cycles for the dual of $\mc{N}$.

\begin{proposition}
\label{cohomologybijection}
Let $Y$ be a smooth projective $k$-curve, and let $N$ be a
constructible $\mb{F}_{p^r}$-\'etale sheaf on $Y$ whose sections
all have open support.  Let
\begin{eqnarray}
\mc{N} = \sheafHom_{\mb{F}_{p^r}} \left( N ,
\mc{O}_Y \right),
\end{eqnarray}
and let $\mc{N}_0$ be a root for $\mc{N}$.  Then the maps
\begin{eqnarray}
H^i \left( Y , N \right) \to H^i \left( Y , 
\mc{N}_0^\vee \right)
\end{eqnarray}
(given by the double-dual homomorphism $N \to \mc{N}_0^\vee$) map
the elements of $H^i \left( Y , N \right)$ bijectively onto
the $F^r$-invariant elements of $H^i \left( Y , \mc{N}_0^\vee \right)$.
\end{proposition}

\begin{proof}
The module $\mc{N}$ is the dual of $N$ under the Riemann-Hilbert correspondence.
By Theorem~\ref{globalexactsequence}, there is an exact sequence
\begin{eqnarray}
\xymatrix{
0 \ar[r] & N \ar[r] & \mc{N}_0^\vee \ar[r]^{1 - F^r} & \mc{N}_0^\vee \ar[r] & 0 }
\end{eqnarray}
which determines an exact sequence of cohomology groups
\begin{eqnarray}
\xymatrix{
0 \ar[r] & H^0 (Y, N) \ar[r] & H^0 ( Y, \mc{N}_0^\vee ) \ar[r]^{1 - F^r} & 
H^0 ( Y, \mc{N}_0^\vee) \\
\ar[r] & H^1 ( Y, N ) \ar[r] & H^1 ( Y, \mc{N}_0^\vee ) \ar[r]^{1 - F^r} &
H^1 ( Y , \mc{N}_0^\vee ) \ar[r] & 0. }
\end{eqnarray}
The vector space $H^0 \left( Y , \mc{N}_0^\vee \right)$ is a trivial left $k[F^r]$-module
(by Proposition~\ref{sepclosedfrobenius}), and it is easily seen that 
the action of $(1 - F^r)$ on this module is surjective.  So this long exact
sequence breaks up into two short exact sequences.  The result follows.
\end{proof}

Proposition~\ref{cohomologybijection} will now enable us to prove the main
result of this paper.  We establish the following notation: if $y$
is a closed point of $Y$, let $\mathfrak{C} \left( N_{(y)} \right)$ denote
the minimal root index of the sheaf $N_{(y)}$ (see 
Definition~\ref{minimalrootindexdef}).  Let $\chi ( Y, N )$ denote
the Euler characteristic of $N$:
\begin{eqnarray}
\chi ( Y , N ) = \dim_{\mathbb{F}_{p^r}} H^0 ( Y, N ) 
- \dim_{\mathbb{F}_{p^r}} H^1 ( Y , N ).
\end{eqnarray}
Likewise, if $\mc{Q}$ is a coherent sheaf on $Y$, let 
$\chi ( Y, \mc{Q} )$ denote the Euler characteristic of $\mc{Q}$.

\begin{theorem}
\label{mainthm}
Let $Y$ be a smooth projective $k$-curve, and let $N$
be a constructible \'etale $\mathbb{F}_{p^r}$-sheaf on $Y$ whose sections
all have open support.  Let $n$ be the generic rank of $N$.  Then,
\begin{eqnarray}
\chi ( Y, N ) \geq n \cdot \chi ( Y , \mc{O}_Y ) -
\sum_{y \in Y(k)} \mathfrak{C} \left( N_{(y)} \right).
\end{eqnarray}
\end{theorem}

\noindent {\it Proof.}
Let
\begin{eqnarray}
\mc{N} = \sheafHom_{\mb{F}_{p^r}} \left( N , \mc{O}_Y \right),
\end{eqnarray}
and let $\mc{N}_0 \subseteq \mc{N}$ be the unique minimal root
for $\mc{N}$.  
By Proposition~\ref{cohomologybijection} and the foregoing
discussion,
\begin{eqnarray}
\dim_{\mathbb{F}_{p^r}} H^0 \left( Y , N \right) &  = &
\dim_k H^0 \left( Y , \mc{N}_0^\vee \right)
\end{eqnarray}
and
\begin{eqnarray}
\dim_{\mathbb{F}_{p^r}} H^1 \left( Y , N \right) &  \leq &
\dim_k H^1 \left( Y , \mc{N}_0^\vee \right). 
\end{eqnarray}
Therefore,
\begin{eqnarray}
\chi ( Y , N ) \geq \chi ( Y , \mc{N}_0^\vee ).
\end{eqnarray}

Let $\mc{N}_1 \subseteq \mc{N}$ be the $\mc{O}_X$-submodule
generated by $F^r ( \mc{N}_0 )$.  The quotient 
$\left( \mc{N}_1 / \mc{N}_0 \right)$ is a coherent 
skyscraper sheaf.
By definition, the $k$-dimension of the stalk of $\left( \mc{N}_1 /
\mc{N}_0 \right)$ at $y$ is $(p^r - 1) \mathfrak{C} \left( \mc{N}_{(y)} \right)$. 
Therefore,
\begin{eqnarray}
\deg \mc{N}_1 - \deg \mc{N}_0 = (p^r - 1) \sum_{y \in Y(k)} \mathfrak{C} \left(
\mc{N}_{(y)} \right).
\end{eqnarray}
Meanwhile, the isomorphism $\mc{N}_1 \cong F^{r*}_X \mc{N}_0$ implies that
$\deg \mc{N}_1 = p^r \deg \mc{N}_0$.  Combining these two equalities yields
\begin{eqnarray}
\deg \mc{N}_0 = \sum_{y \in Y(k)} \mathfrak{C} \left(
\mc{N}_{(y)} \right).
\end{eqnarray}
The desired result now follows using the Riemann-Roch formula:
\begin{eqnarray}
\chi ( Y , N ) & \geq & \chi ( Y , \mc{N}_0^\vee ) \\
& = & n \cdot \chi ( Y , \mc{O}_Y ) + \deg \mc{N}_0^\vee \\
& = & n \cdot \chi ( Y , \mc{O}_Y ) - \deg \mc{N}_0 \\
& = & n \cdot \chi ( Y , \mc{O}_Y ) - \sum_{y \in Y(k)}
\mathfrak{C} \left( N_{(y)} \right).   \qed
\end{eqnarray}

\subsection{Examples}

\label{examplessubsection}

To illustrate Theorem~\ref{mainthm}, we compare three different
examples of rank-$2$ sheaves on the projective line.  In the following,
we will assume $p \geq 5$.  If $Z$ is a $k$-curve, let $\underline{\mathbb{F}_p}_Z$
(or simply $\underline{\mathbb{F}_p}$) denote the constant $\mb{F}_p$-sheaf on
$Z$.

Note that if $f \colon W \to W'$ is a finite morphism of smooth projective
$k$-curves, and $Q$ is a constructible $\mathbb{F}_p$-\'etale sheaf
on $W$, then the dimensions of the cohomology groups $H^i \left( W', f_* Q \right)$
are the same as those of $H^i \left( W, Q \right)$.  (This is 
apparent from a Leray-Serre spectral sequence.)

\begin{example}
Consider the open immersion
$j \colon \mathbb{A}^1_k \smallsetminus \{ 0 \} \to \mathbb{P}^1_k$.  Let
\begin{eqnarray}
M = j_! \left( \underline{\mathbb{F}_p}^{\oplus 2} \right).
\end{eqnarray}
The Euler characteristic of $M$ is $-2$.  (This can be proven easily 
with an exact sequence.)  The sheaf
\begin{eqnarray}
\mc{M} = \sheafHom_{\mathbb{F}_p} \left( M , \mc{O}_{\mathbb{P}^1} \right)
\end{eqnarray}
is isomorphic as a left $\mc{O}_{F^r, \mathbb{P}^1}$-module to
$j_* \mc{O}_{\mathbb{A}^1 \smallsetminus \{ 0 \}}^{\oplus 2}$.  Under this isomorphism,
the minimal root $\mc{M}_0 \subseteq \mc{M}$ can be identified with the set of
sections that have poles of order at most $1$ at both
$0$ and $\infty$.  If $\mc{M}_1 \subseteq \mc{M}$ is the subsheaf
generated by $F^r \left( \mc{M}_0 \right)$, then
\begin{eqnarray}
\dim_k \left( \mc{M}_1 / \mc{M}_0 \right)_0 =
\dim_k \left( \mc{M}_1 / \mc{M}_0 \right)_\infty = 2(p - 1).
\end{eqnarray}
The minimal root index for $M$ at both $0$ and $\infty$ is $2$.
In this case the formula from Theorem~\ref{mainthm} yields
$\chi ( \mathbb{P}^1 , M )$ exactly:
\begin{eqnarray}
2 \cdot \chi ( \mathbb{P}^1 , \mc{O}_{\mathbb{P}^1} ) -
\mathfrak{C} \left( M_{(0)} \right) - \mathfrak{C}
\left( M_{(\infty)} \right) = 2 \cdot 1 - 2 - 2 = -2.
\end{eqnarray}
\end{example}

\begin{example}
\label{deg2example}
Let $f \colon \mathbb{P}^1 \to \mathbb{P}^1$ be a degree-$2$ morphism
which maps $0$ to $0$ and $\infty$ to $\infty$ and is ramified at both
of those points.  Let $N = f_* \underline{\mathbb{F}_p}$.
Then $\chi ( \mathbb{P}^1 , N ) = \chi ( \mathbb{P}^1 , \mathbb{F}_p )
= 1$.

The sheaf $N$ is locally constant away from the ramified points of $f$.
So for any closed point $x \notin \{ 0, \infty \}$, the local sheaf
$N_{(x)}$ is isomorphic to $\left( \underline{\mathbb{F}_p } \right)^{\oplus 2}$.  
The sheaf $N_{(\infty)}$ is a nontrivial rank-$2$ sheaf which can be
trivialized by a quadratic extension of $\mc{O}_{\mathbb{P}^1, \infty}$.
The reader may verify that
there is a simple decomposition
\begin{eqnarray}
N_{(\infty)} \cong \left( \underline{\mathbb{F}_p} \right) \oplus T,
\end{eqnarray}
where $T$ is the nontrivial rank-$1$ sheaf which arose in Example~\ref{quadraticexample}.
By the calculation in that example (and by the fact that the minimal root
index is clearly additive over direct sums), we find
\begin{eqnarray}
\mathfrak{C} \left( N_{(\infty)} \right) = \frac{1}{2}.
\end{eqnarray}
A similar calculation shows that the minimal root index of
$N_{(0)}$ is $\frac{1}{2}$.  So the formula from
Theorem~\ref{mainthm} yields
\begin{eqnarray}
2 \cdot \chi ( \mathbb{P}^1 , \mc{O}_{\mathbb{P}^1} ) -
\mathfrak{C} \left( N_{(0)} \right) - \mathfrak{C}
\left( N_{(\infty)} \right) = 2 \cdot 1 - \frac{1}{2}
- \frac{1}{2} = 1,
\end{eqnarray}
which is equal to $\chi ( \mathbb{P}^1 , N )$.
\end{example}

\begin{example}
\label{ellipticcurveexample}
Let $E$ be an elliptic curve, and suppose that $g \colon 
E \to \mathbb{P}^1$ is a degree-$2$ morphism which is
ramified at $4$ distint points in $\mathbb{P}^1$.  Let
$P = g_* \left( \underline{\mathbb{F}_p} \right)$.  Let
$a_1, a_2, a_3, a_4 \in \mathbb{P}^1$ be the ramified points
of $g$.  A calculation similar to the one in Example~\ref{deg2example}
shows that
\begin{eqnarray}
\mathfrak{C} \left( P_{(a_i)} \right) = \frac{1}{2}.
\end{eqnarray}
Note that $\chi ( \mathbb{P}^1 , P )$ is equal to $\chi ( E , \mathbb{F}_p )$,
which can be $0$ or $1$, depending on whether $E$ is supersingular.
The lower bound for $\chi ( \mathbb{P}^1 , P )$ given by
Theorem~\ref{mainthm} is
\begin{eqnarray}
2 \cdot \chi \left( \mathbb{P}^1 ,
\mathcal{O}_{\mathbb{P}^1} \right) - \sum_{i=1}^4 \mathfrak{C}
\left( P_{(a_i)} \right) = 0.
\end{eqnarray}
So equality occurs in this case if and only if
$E$ is an ordinary elliptic curve.
\end{example}

\end{document}